\documentclass[a4paper,reqno,11pt]{amsart}
\usepackage{amsfonts}
\usepackage{amsaddr}
\usepackage{hyperref}
\usepackage[left=3cm,right=3cm,top=3cm,bottom=2.7cm,headsep=0.7cm]{geometry}
\usepackage[T1]{fontenc} 
\usepackage[utf8]{inputenc}
\usepackage{amsmath,amssymb,amsthm,mathrsfs,esint}
\usepackage{mathtools}

\newcommand{\R}{{\mathbb R}}
\newcommand{\A}{{\mathbb A}}
\newcommand{\C}{{\mathbb C}}
\newcommand{\N}{{\mathbb N}}

\newcommand{\E}{{\mathrm{E}}}
\newcommand{\m}{{\textbf{m}}}
\newcommand{\oa}{{\otimes_{\A}}}

\newcommand{\Rn}{{\R}^n}

\newcommand{\dx}{\, \mathrm{d} x}
\newcommand{\dd}{ {\, \mathrm{d}} }
\renewcommand{\dh}{\, \mathrm{d} \mathcal{H}^{n-1}}
\newcommand{\hn}{\mathcal{H}^{n-1}}

\newcommand{\Sn}{{\mathbb{S}^{n-1}}}
\newcommand{\ol}{\overline}

\newcommand{\tr}{\mathrm{tr}\,}
\newcommand{\tro}{\mathrm{tr}_\dom}
\newcommand{\sm}{\setminus}
\newcommand{\Mnn}{{\mathbb{M}^{n\times n}_{sym}}}
\newcommand{\dod}{{\partial_D \Omega}}
\newcommand{\don}{{\partial_N \Omega}}
\newcommand{\dom}{{\partial \Omega}}
\newcommand{\weak}{\rightharpoonup}
\newcommand{\wstar}{\stackrel{*}\rightharpoonup}

\newcommand{\mres}{\mathbin{\vrule height 1.6ex depth 0pt width
0.13ex\vrule height 0.13ex depth 0pt width 1.3ex}}

\DeclareMathOperator*{\aplim}{ap\,lim}

\theoremstyle{plain}
\begingroup
\theoremstyle{plain}
\newtheorem{theorem}{Theorem}[section]
\newtheorem{corollary}[theorem]{Corollary}
\newtheorem{proposition}[theorem]{Proposition}

\theoremstyle{definition}
\newtheorem{definition}[theorem]{Definition}
\theoremstyle{remark}
\newtheorem{remark}[theorem]{Remark}

\endgroup

\numberwithin{equation}{section}
\title[Approximation of cohesive 
fracture energies with an activation threshold]{Phase-field approximation for a class of cohesive fracture energies with an activation threshold} 
\author{Antonin Chambolle \and Vito Crismale}
\address{CMAP, \'Ecole Polytechnique, 91128 Palaiseau Cedex, France}
\email[Antonin Chambolle]{antonin.chambolle@cmap.polytechnique.fr}
\email[Vito Crismale]{vito.crismale@polytechnique.edu}

\begin{document}
\begin{abstract} We study the $\Gamma$-limit of Ambrosio-Tortorelli-type functionals $D_\varepsilon(u,v)$, whose dependence on the symmetrised gradient $e(u)$ is different in $\A u$  and in $e(u)-\A u$, for a $\C$-elliptic symmetric operator $\A$, in terms of the prefactor depending on the phase-field variable $v$. 
The limit energy depends both on the opening and on the surface of the crack, and is intermediate between the Griffith brittle fracture energy and the one considered by Focardi and Iurlano in \cite{FocIur14}.
In particular we prove that $G(S)BD$ functions with bounded $\A$-variation are $(S)BD$.
\smallskip

\noindent
\end{abstract}

\maketitle

\tableofcontents

\section{Introduction}\label{sec:intro}

The energy functionals in Fracture Mechanics are usually expressed  in terms of the \emph{displacement} $u\colon \Omega \subset \Rn \to \Rn$ 
as the sum of a volume part, accounting for the mechanical properties of the uncracked material in the bulk region, and of a surface part, concentrated on  a 
$(n{-}1)$-dimensional discontinuity set of $u$ (the \emph{crack set})
and representing the energy dissipated in the crack process. 

The presence of the crack set entails difficulties in the effective computation of minimisers, for instance by numerical simulations. A possible, and by now classical, way out is to approximate the energy in the sense of $\Gamma$-convergence, through simpler functionals. These depend on the two variables $u\colon \Omega\to \Rn$, which is now a Sobolev function and represents the regularised displacement, and the phase-field $v\colon \Omega \to [0,1]$, whose sublevels $\{v < s\}$, for $s\in (0,1)$, may be used to approximate the limit discontinuity set.
%
%
Such approximations are often called of Ambrosio-Tortorelli type, from the breakthrough paper \cite{AmbTorCPAM} they realised to approximate the Mumford-Shah functional \cite{MumSha} in image reconstruction. 

In the context of Fracture Mechanics,  Ambrosio-Tortorelli approximations are largely employed since the Francfort-Marigo's work \cite{FraMar98} on the variational approach to fracture and the first numerical experiments \cite{BouFraMar00} (see e.g.\ \cite{Bou07, AmbGerDeL15, AAGVD} and references therein). The first case that has been considered is the Griffith energy \cite{Griffith} 
\begin{equation}\tag{G}\label{eq:Griffith}
\int \limits_\Omega f_p(e(u)) \dx + \hn\Big(J_u \cup \big(\dod \cap \{\tro u\neq \tro u_0\}\big)\Big)\,
\end{equation}
where $u_0 \in W^{1,p}(\Rn;\Rn)$ enforces a Dirichlet boundary condition (by penalising $\mathrm{tr}_\dom u$, the trace on $\dom$ of $u$, where different from that of $u_0$  on the Dirichlet boundary $\dod$),  
$e(u):=\frac{\nabla u + \nabla u^T}{2} \in \Mnn$ is the \emph{linearised strain} (in the bulk)  in \emph{small strain assumptions}, $J_u$ is the \emph{jump set}
of $u$ (see Section~\ref{sec:notation}), $\hn$ is the $(n{-}1)$-dimensional Hausdorff measure, and $f_p \colon \Mnn \to [0, +\infty)$ is convex with
\begin{equation} \label{2607182230}\tag{HP$1\,f_p$}
 f_p(0)=0,\qquad  C_{f_p} (|\xi|^p -1) <f_p(\xi)<  C'_{f_p} (|\xi|^p +1)\,, \quad p>1\,,
 \end{equation} 
 for the Frobenius norm $|\cdot|$ 
 on $\Mnn$. As explained e.g.\ in \cite[Section~1]{CFI17DCL} and \cite[Sections~10 and 11]{Hut}, the reference form for $f_p$ is for every $\mu>0$
  \begin{equation}\label{eq:fpm}
f_{p,\mu}(\xi):= \frac{1}{p} \big( (\Sigma \xi \colon \xi + \mu)^{\frac{p}{2}} - \mu^{\frac{p}{2}} \big)\,, 
\end{equation}
where $\Sigma$, such that
$\Sigma(\xi-\xi^T)=0$ and $\Sigma \xi \cdot \xi \geq c_0\, |\xi+\xi^T|^2$ for all $\xi \in \Mnn$,
is the fourth-order Hooke's
tensor: this is a slight generalisation of the original Griffith energy, where the bulk energy is the \emph{linear elastic} energy, that is $p=2$, $\mu=0$ and 
\[
\Sigma \xi \cdot \xi= \frac{1}{4} \lambda_1 |\xi+\xi^T|^2 + \frac{1}{2} \lambda_2 (\mathrm{Tr\,}\xi)^2\,,
\] with $\lambda_1$, $\lambda_2$ the Lamé coefficients. 
The $f_{p,\mu}$ are quadratic for small $\xi$ and with $p$-growth for large $\xi$, and for $p\neq 2$ this may 
account for plastic deformation at large strain.

The Griffith energy \eqref{eq:Griffith} is approximated by the functionals
\begin{equation}\label{2807182146}\tag{G$_\varepsilon$}
\int \limits_\Omega \Big((v+\eta_\varepsilon) f_p(e(u))+\frac{(1-v)^2}{4 \varepsilon}+\varepsilon^{q-1} |\nabla v|^q \Big)\dx \,, \qquad \lim_{\varepsilon \to 0} \frac{\eta_\varepsilon}{\varepsilon^{ p-1 }}=0\,,
\end{equation}
for
 $u\in W_{u_0}^{1,p}(\Omega; \Rn):= W^{1,p}(\Omega; \Rn) \cap \{u \colon \tr_\dom (u - u_0)=0 \text{ on }\dod\}$,
  $v \in W_1^{1,q}(\Omega;[0,1]):= W^{1,q}(\Omega;[0,1]) \cap \{v \colon \tr_\dom v= 1 \text{ on }\dod\}$, 
 and $+\infty$ otherwise:
such approximation has been proven without any \emph{a priori} assumption on $u$, for any $p>1$, and in any dimension in \cite{CC17}, together with compactness for minimisers (see \cite{CC18}),  assuming that
\begin{equation}\label{1806181920}
O_{\delta,x_0}(\dod) \subset \Omega \qquad \text{for $\delta \in (0,\ol \delta)$,}
\end{equation}
for some $\ol \delta>0$ and $x_0\in \Rn$, where $O_{\delta,x_0}(x):=x_0+(1-\delta)(x-x_0)$. 
 This generalises \cite{Cha04, Cha05Add, Iur14}, assuming \emph{a priori} $u \in L^2$ and $p=2$, \cite{CFI17Density}, requiring $u\in L^p$, $p>1$, and \cite{FriPWKorn}, obtained in dimension 2 (see also e.g.\ \cite{Gia05, BurOrtSul13, Neg06}
 for the antiplane shear case and different approximations).

In \cite{FocIur14} 
Focardi and Iurlano studied the 
 limit 
of the functionals 
\begin{equation}\label{2807182147}\tag{C$_\varepsilon$}
\int \limits_\Omega \Big((v+\varepsilon) f_{2,0}(e(u))+\frac{\psi(v)}{\varepsilon}+\,\varepsilon^{q-1} |\nabla v|^q \Big)\dx\,,
\end{equation}
for $u\in H^1(\Omega;\Rn)$, $v\in W^{1,q}(\Omega;[0,1])$, and $+\infty$ otherwise (with $\psi\in C([0,1])$ decreasing, $\psi(1)=0$) and proved that they $\Gamma$-converge to
\begin{equation}\label{en:FocIur}\tag{C}
\int \limits_\Omega  f_{2,0}(e(u)) \dx + c_1 \hn(J_u) + c_2 \int \limits_{J_u} \big| [u] \odot \nu_u \big| \dh\,,
\end{equation}
for suitable $c_1$, $c_2>0$.
The energy space for \eqref{en:FocIur} is $SBD^2$, a subspace of the Special Bounded Deformation functions $SBD$ (see Section~\ref{sec:notation}). For $v \in SBD$, the distributional gradient $\E v:= \frac{\mathrm{D}v + \mathrm{D}^T v}{2}$ is a bounded Radon measure, $J_v$ is the set of points $x$ at which $v$ has two different approximate limits $v^+(x)$,  $v^-(x)$ 
with respect to a suitable direction $\nu_v(x)$, and $[v](x):=v^+(x) - v^-(x)$ is the \emph{jump}. We denote by $\odot$ the symmetrised tensor product, and notice that $[u]\odot \nu_u$ is the part of the total strain $\E u$ concentrated on $J_u$, see \eqref{1205181701}.

The energy \eqref{en:FocIur} 
depends also on the jump amplitude, reflecting mechanical interaction between the fracture lips.  This is typical of \emph{cohesive} fracture energies, in contrast to the \emph{brittle} energy \eqref{eq:Griffith}. On the other hand, \eqref{en:FocIur} has not the form of the classical  cohesive fracture energies in Barenblatt's model \cite{Bar62}, which in particular do not depend on $\hn(J_u)$.  The presence of the measure of the crack surface corresponds to an activation energy which is necessary to nucleate the crack: this
is considered also in \cite{AleMarMauVid}, where it is called ``depinning energy'', in \cite{Almi}, that studies a model for quasistatic evolution, and in the approximation result \cite{BarLazZep16}.
A few others have succeeded in approximating particular instances of pure cohesive energies, as
in \cite{Iur13, CFI15, DMOrlToa16}, see also \cite{AliBraSha99} (in these works the bulk energy is a function of the full gradient $\nabla u$). 

 In this work we approximate fracture energies 
that, as \eqref{en:FocIur}, include the measure of $J_u$, but whose 
cohesive term now 
depends
only on a part of the strain, for instance on its deviatoric part (for $n\geq 3$). Moreover, we consider general $p$-growth ($p>1$) in $e(u)$ of the bulk energy, no integrability assumptions on $u$, and study the Dirichlet boundary problem. 
To present the general case we consider a constant-coefficient, linear, first order differential operator 
\begin{equation}\label{1806182316}
\A u = \sum_{j=1}^n A_j \partial_j u\,, \qquad u \colon \Rn \to \Rn\,,
\end{equation}
for $A_j\in \mathcal{L}(\Rn,\mathbb{M}^{n{\times}n})$ linear mappings.  We assume that $(A_j)_i=(A_i)_j$ and that $\A u \colon \Rn \to \Mnn$, so that  
there is an endomorphism $A$ of $\Mnn$ for which
\begin{equation}\label{2406180949}
\A u = A(e(u))\,.
\end{equation} 
A \emph{Fourier symbol mapping} $\A[z] \colon \Rn \to \Mnn$ is introduced for every $z=(z_1, \dots, z_n)\in \Rn$, 
defined by
\begin{equation}\label{1906181746}
\A[z]v:= v \, \oa z:= \sum_{j=1}^n z_j A_j v = A([v]\odot z)
\end{equation}
for $v\in \Rn$; the operator $\A$ is 
$\R$-\emph{elliptic} if $\A[z]$ is injective for all $z \in \R^n \sm \{0\}$, and  $\C$-\emph{elliptic} if (take the estension of $\A[z]v$ on $\C^n$) $\A[z]\colon \C^n \to \C^{n{\times}n}$ is injective for all $z \in \C^n \sm \{0\}$. These operators have been recently considered in e.g.\ \cite{BreDieGme17, GmeRai17, GmeRai18, DePRin16, ArDeHiRi18, SpeVS18}.
The deviator operator $\E_D u:= \E u -\frac{1}{n}\,(\mathrm{div\,}u) \, \mathrm{Id}_n$ is $\C$-elliptic for $n\geq 3$, but not for $n=2$ (see Remark~\ref{rem:1906181942}).

 From a mechanical point of view, the reference problem is to minimise the energy $F$ under a Dirichlet boundary condition on a part of the boundary $\dod \neq \emptyset$ 
  with possibly the presence of volume forces, and surface forces on the remaining part of the boundary  $\don$, with 
  \begin{equation}\label{3101190848}
  \dom=\dod\cup \don\cup N\,, \quad \dod \cap \don =\emptyset\,, \quad \mathcal{H}^{n-1}(N)=0\,,  \quad\partial(\dod)=\partial(\don)\,,
  \end{equation}
  for $\dod$ and $\don$ relatively open.
Here we assume all forces null, \eqref{1806181920}, and that 
\begin{equation}\label{1806181046}\tag{HP$2\,f_p$}
\hspace{-2em}\lim_{s\to \pm \infty} \frac{f_p(s\, \xi)}{|s|^p}= \tilde{f}_p(\xi) \quad\text{ uniformly as $s\to \pm\infty$, $\xi \in \Mnn$.}
\end{equation}
We have that $\tilde{f}_p$ is positively $p$-homogeneous, and $(\tilde{f}_p)^{\frac{1}{p}}$ is a norm on $\Mnn$ (cf.\ e.g. \cite[Remark~2.7]{FonFus97}).
Then we prove the following, main result of this work. 
 \begin{theorem}\label{teo:main}
Let $\Omega\subset \Rn$ be open bounded Lipschitz  satisfying \eqref{1806181920}, \eqref{3101190848},
$u_0\in W^{1,p}(\Rn; \Rn)$, $p$, $q > 1$, $\gamma>0$, $\varepsilon>0$, $\eta_\varepsilon>0$ such that $\lim_{\varepsilon \to 0} \frac{\eta_\varepsilon}{\varepsilon^{p-1}}=0$, $f_p$ and $\tilde{f}_p$ satisfying \eqref{2607182230}, \eqref{1806181046}, 
$\psi \in C([0,1])$ decreasing with $\psi(1)=0$, 
and
$\A$ be $\C$-elliptic.
 Then the functionals $D_\varepsilon(u, v)$ defined on $u$, $v$ measurable by
\begin{equation*}
D_\varepsilon(u,v):= \int \limits_\Omega \Big[ (v+ \varepsilon^{p-1}) f_p(\A u) + (v + \eta_\varepsilon)\, f_p(e(u)- \A u) + \frac{\psi(v)}{\varepsilon} + \gamma \varepsilon^{q-1}|\nabla v|^q \Big] \dx\,
\end{equation*}
if $u\in W_{u_0}^{1,p}(\Omega; \Rn)$,  $v \in W_1^{1,q}(\Omega;[0,1])$
and by $+\infty$ otherwise, $\Gamma$-converge, as $\varepsilon\to 0$, to 
\begin{equation*}
\begin{split}
D(u,v):=
\int \limits_\Omega & \Big[ f_p\big(A(e(u))\big) +  \, f_p\big(e(u)-A(e(u))\big) \Big] \dx + \int \limits_{J_u} \big[ a + b\, (\tilde{f}_p)^{\frac{1}{p}}([u] \oa \nu_u) \big] \dh \\&
+ \hspace{-2.5em}\int \limits_{\dod \cap \{\mathrm{tr}_{\dom}(u - u_0) \neq 0\}} \hspace{-3em} \big[ a + b\, (\tilde{f}_p)^{\frac{1}{p}}\big( \mathrm{tr}_{\dom} (u - u_0) \oa \nu_{\dom} \big) \big] \dh\,,
\end{split}
\end{equation*}
if
\[
u\in SBD^p(\Omega)\,,\quad  v=1 \text{ a.e.\ in }\Omega\,,
\]
and by $+\infty$ otherwise for $u$, $v$ measurable,
with respect to the topology of convergence in $\mathcal{L}^n$-measure for $u$ and $v$. Above $A$ is the operator introduced in \eqref{2406180949}, and ($\frac{1}{p'}+\frac{1}{p}=\frac{1}{q'}+\frac{1}{q}=1$)
\[a:= 2 (q')^{1/q'} (\gamma q)^{1/q} \int_0^1 \psi^{1/q'},\qquad
b:= p^{1/p} (p')^{1/p'} \psi(0)^{1/p}\,.
\]
Moreover, for every $M>0$ and $\varepsilon < 1$, the sublevel $\{(u,v)\colon D_\varepsilon(u,v) \leq M\}$ is contained in
\[
  \Big\{(u,v) \colon \int_\Omega |\A u| \dx \leq C_M,\, \mathrm{tr}_{\dom}\, u= \mathrm{tr}_{\dom}\, u_0 \text{ on }\dod,\, \int_\Omega \psi(v) \dx\leq  M \varepsilon \Big\}\,.
\] 
Then a sequence of quasi-minimisers for $D_\varepsilon$ converge, up to a subsequence, to a minimiser of $D$,
with respect to the product of the strong $L^r(\Omega;\Rn)$ topology for $u$, for 
any 
$r \in [1, \frac{n}{n-1})$, times the topology of convergence in $\mathcal{L}^n$-measure 
for $v$,  
provided $\hn(\dod \cap \partial \Omega_j)>0$, for each $\Omega_j$ connected component of $\Omega$.
 \end{theorem}

The functionals $f_{p,\mu}$ in \eqref{eq:fpm} satisfy \eqref{1806181046}, and $f_{2,\mu}(\A u) + f_{2,\mu}(e(u)- \A u)= f_{2,\mu}(e(u))$ if $\A = \E_D$, so, in this case, we recover the linear elastic energy in the bulk.

Our approximating functionals are in some sense intermediate between those in \eqref{2807182146} and \eqref{2807182147}, since the part corresponding to $e(u)-\A u$ is multiplied by $v+\eta_\varepsilon$ as in \eqref{2807182146}, while $f_p(\A u)$ is multiplied by  $v+\varepsilon^{p-1}$ as in \eqref{2807182147} for $p=2$. This results in an interaction between $(v+ \varepsilon^{p-1}) f_p(\A u)$ and $\frac{\psi(v)}{\varepsilon}$ that gives the 
term in $[u]$ in the limit. As usual, the surface of $J_u$ is approximated by the Ambrosio-Tortorelli part $ \frac{\psi(v)}{\varepsilon} + \gamma \varepsilon^{q-1}|\nabla v|^q$.

Since the integrals of $(v+\eta_\varepsilon) f_p(e(u)-\A u)$ and $\frac{\psi(v)}{\varepsilon}$ are not energetically of the same order, we have an \emph{a priori} control only on $\A u$ as a Radon measure, differently from \cite{FocIur14}, where this control is on the whole $\E u$. For this reason we initially work in the space $GSBD$ of \emph{generalised} $SBD$ functions, introduced by Dal Maso in \cite{DM13} to study brittle fracture  (in \cite{FocIur14} the control on $\E u$ allowed to work directly in $SBD$). 
A crucial point is to establish the expression of $\A u$ on the set $J_u$, in particular to show that \[\A u \mres J_u= [u] \oa \nu_u \, \hn \mres J_u\] ($J_u$, $[u]$, $\nu_u$ are well defined in $GSBD$, see Section~\ref{sec:notation}): we prove this equality employing the tools developed in \cite{BreDieGme17, GmeRai17, GmeRai18}  to show the existence of a trace for functions with bounded $\A$-variation if and only if $\A$ is $\C$-elliptic. This is enough to conclude that $GSBD$ functions with bounded $\A$-variation are in fact in $SBD$, because we deduce that $[u]$ is integrable on $J_u$ 
(this is also true for $GBD$, $BD$ in place of $GSBD$, $SBD$, see Theorem~\ref{teo:GSBDdiventaSBD}).
 Technical problems arise to prove the same in dimension 2 for $\A=\E_D$ (see Remark~\ref{2907181935}). 

We remark that in \cite{CCF17} the approximating functionals weight differently $\E_D u + \mathrm{div}^+ u$ (multiplied by $(v+\epsilon)$) and $\mathrm{div}^- u$ (without any prefactor), so that $\mathrm{div}^- u$ is equibounded in $L^2$ and in the limit $[u]\cdot \nu \geq 0$, a linearised non-interpenetration condition. Here we could also separate the behaviours of $\mathrm{div}^+ u$ and $\mathrm{div}^- u$ (Remarks~\ref{rem:0208182013} and \ref{rem:0308182234}), but the meaning of our approach is in some sense opposite to \cite{CCF17}, since we do not pay, in the limit 
part in $[u]$,
for 
the terms multiplied by $v+\eta_\varepsilon$, while in \cite{CCF17} the  concentration of terms without prefactor pay infinite energy. One might also consider a non-interpenetration condition in our model, for instance by studying the $\Gamma$-limit of (for $\mathrm{Id}$ the identity $n{\times}n$ matrix)
\begin{equation*}
\int \limits_\Omega \Big[ (v+ \varepsilon^{p-1}) f_p(\E_D u) + (v + \eta_\varepsilon)\, f_p(\mathrm{div}^+ u \,\mathrm{Id}) + f_p(\mathrm{div}^- u \,\mathrm{Id}) + \frac{\psi(v)}{\varepsilon} + \gamma \varepsilon^{q-1}|\nabla v|^q \Big] \dx\,,
\end{equation*}
but the $\Gamma$-$\limsup$ inequality presents hard difficulties. In this respect, we point out that for the $\Gamma$-$\limsup$ inequality in Theorem~\ref{teo:main} we employ the approximation result \cite{Cri19} for $SBD^p$ functions in $BD$-norm, which allows us to prove the result without any regularity assumption on the displacement, as the uniform $L^\infty$ bound in \cite{FocIur14}.  We have also to refine the argument of \cite{Cri19}, in order to deal with Dirichlet boundary conditions. 



\section{Notation and a preliminary result}\label{sec:notation}

We denote by $\mathcal{L}^n$ and $\mathcal{H}^k$ the $n$-dimensional Lebesgue measure and the $k$-dimensional Hausdorff measure. For any locally compact subset $B$ of $\Rn$, the space of bounded $\R^m$-valued Radon measures on $B$ is indicated as $\mathcal{M}_b(B;\R^m)$. For $m=1$ we write $\mathcal{M}_b(B)$ for $\mathcal{M}_b(B;\R)$ and $\mathcal{M}^+_b(B)$ for the subspace of positive measures of $\mathcal{M}_b(B)$. For every $\mu \in \mathcal{M}_b(B;\R^m)$, $|\mu|(B)$ stands for its total variation.
We use the notation: 
$B_\varrho(x)$ [and $Q_\varrho(x)$] for the open ball [cube] with center $x$ and radius [sidelength] $\varrho$; $x\cdot y$, $|x|$ for the scalar product and the norm in $\Rn$; $1^*$ for $n/(n-1)$, $n$ being the space dimension; $\dd(x, A)$ for the distance of $x$ from $A$; $A \Subset K$ when $A$ compactly contained in $K$.

We recall the definition of approximate limit with respect to the convergence in measure and approximate jump set for measurable functions.
\begin{definition}\label{def:aplim}
Let $A\subset \Rn$, $v\colon A \to \R^m$ an $\mathcal{L}^n$-measurable function, $x\in \Rn$ such that
\begin{equation*}
\limsup_{\varrho\to 0^+}\frac{\mathcal{L}^n(A\cap B_\varrho(x))}{\varrho^n}>0\,.
\end{equation*}
A vector $a\in \Rn$ is the \emph{approximate limit} of $v$ as $y$ tends to $x$ if for every $\varepsilon>0$
\begin{equation*}
\lim_{\varrho \to 0^+}\frac{\mathcal{L}^n(A \cap B_\varrho(x)\cap \{|v-a|>\varepsilon\})}{\varrho^n}=0\,,
\end{equation*}
and then we write
\begin{equation}\label{3105171542}
\aplim \limits_{y\to x} v(y)=a\,.
\end{equation}
\end{definition}

\begin{definition}\label{def:2306181018}
Let $U\subset \Rn$ open, and $v\colon U\to \R^m$ be $\mathcal{L}^n$-measurable. The \emph{approximate jump set} $J_v$ is the set of points $x\in U$ for which there exist $a$, $b\in \R^m$, with $a \neq b$, and $\nu\in \Sn$ such that
\begin{equation*}
\aplim\limits_{(y-x)\cdot \nu>0,\, y \to x} v(y)=a\quad\text{and}\quad \aplim\limits_{(y-x)\cdot \nu<0, \, y \to x} v(y)=b\,.
\end{equation*}
The triplet $(a,b,\nu)$ is uniquely determined up to a permutation of $(a,b)$ and a change of sign of $\nu$, and is denoted by $(v^+(x), v^-(x), \nu_v(x))$. The jump of $v$ is the function 
defined by $[v](x):=v^+(x)-v^-(x)$ for every $x\in J_v$. 
\end{definition}
\par
\medskip
\paragraph{\bf $BV$ and $BD$ functions.}
For $U\subset \Rn$ open, a function $v\in L^1(U)$ is a \emph{function of bounded variation} on $U$, denoted by $v\in BV(U)$, if $\mathrm{D}_i v\in \mathcal{M}_b(U)$ for $i=1,\dots,n$, where $\mathrm{D}v=(\mathrm{D}_1 v,\dots, \mathrm{D}_n v)$ is its distributional gradient. A vector-valued function $v\colon U\to \R^m$ is $BV(U;\R^m)$ if $v_j\in BV(U)$ for every $j=1,\dots, m$.

A $\mathcal{L}^n$-measurable bounded set $E\subset \R^n$ is a set of \emph{finite perimeter} if $\chi_E$ is a function of bounded variation. The \emph{reduced boundary} of $E$, denoted by $\partial^*E$, is the set of points $x\in \mathrm{supp}\, |\mathrm{D}\chi_E|$ such that the limit $\nu_E(x):=\lim_{\varrho \to 0^+}\frac{\mathrm{D}\chi_E(B_\varrho(x))}{|\mathrm{D}\chi_E|(B_\varrho(x))}$ exists and satisfies $|\nu_E(x)|=1$. The reduced boundary is countably $(\hn, n-1)$ rectifiable, and the function $\nu_E$ is called \emph{generalised inner normal} to $E$.

The space of \emph{functions of bounded deformation} on $U$ is
\begin{equation*}
BD(U):=\{v\in L^1(U;\Rn) \colon \E v \in \mathcal{M}_b(U;\Mnn)\}\,,
\end{equation*}
where $\E v$ is the distributional symmetric gradient of $v$.
It is well known (see \cite{AmbCosDM97, Tem}) that for $v\in BD(U)$, the \emph{jump set} $J_v$, defined as the set of points $x\in U$ where $v$ has two different one sided Lebesgue limits $v^+(x)$ and $v^-(x)$ with respect to a suitable direction $\nu_v(x)\in \Sn$, is countably $(\hn, n-1)$ rectifiable (see, e.g.\ \cite[3.2.14]{Fed}), and that
\begin{equation*}
\mathrm{E}v=\mathrm{E}^a v+ \mathrm{E}^c v + \mathrm{E}^j v\,,
\end{equation*}
where $\mathrm{E}^a v$ is absolutely continuous with respect to $\mathcal{L}^n$, $\mathrm{E}^c v$ is singular with respect to $\mathcal{L}^n$ and such that $|\mathrm{E}^c v|(B)=0$ if $\hn(B)<\infty$, while 
\begin{equation}\label{1205181701}
\E^j v=[v]\odot \nu_v \,\hn  \mres  J_v\,.
\end{equation}
In the above expression of $\E^j v$, $[v]$ denotes the \emph{jump} of $v$ at any $x\in J_v$ and is defined by $[v](x):=(v^+-v^-)(x)$, the symbols $\odot$ and $\mres$ stands for the symmetric tensor product and the restriction of a measure to a set, respectively. Since $|a\odot b| \geq |a||b|/\sqrt{2}$ for every $a$, $b$ in $\Rn$, it holds $[v]\in L^1(J_v;\Rn)$. 
The density of $\mathrm{E}^a v$ with respect to $\mathcal{L}^n$ is denoted by $e(v)$, and we have that (see \cite[Theorem~4.3]{AmbCosDM97}) for $\mathcal{L}^n$-a.e.\ $x\in U$
\begin{equation}\label{1906180942}
\aplim_{y\to x}\frac{\big(v(y)-v(x)-e(v)(x)(y-x)\big)\cdot (y-x)}{|y-x|^2} =0\,.
\end{equation}
The space $SBD(U)$ is the subspace of all functions $v\in BD(U)$ such that $\mathrm{E}^c v=0$, while for $p\in (1,\infty)$
\begin{equation*}
SBD^p(U):=\{v\in SBD(U)\colon e(v)\in L^p(U;\Mnn),\, \hn(J_v)<\infty\}\,.
\end{equation*}
Analogous properties hold for $BV$, as the countable rectifiability of the jump set and the decomposition of $\mathrm{D}v$, and the spaces $SBV(U;\R^m)$ and $SBV^p(U;\R^m)$ are defined similarly, with $\nabla v$, the density of $\mathrm{D}^a v$ with respect to $\mathcal{L}^n$, in place of $e(v)$.

We now recall some slicing properties of $SBD$ that will be useful in Theorem~\ref{teo:GSBDdiventaSBD}. 
As general notation, fixed $\xi \in \Sn:=\{\xi \in \Rn\colon |\xi|=1\}$, for any $y\in \Rn$ and $B\subset \Rn$ let
\begin{equation*}
\Pi^\xi:=\{y\in \Rn\colon y\cdot \xi=0\},\qquad B^\xi_y:=\{t\in \R\colon y+t\xi \in B\}\,,
\end{equation*}
and for every function $v\colon B\to \R^n$ and $t\in B^\xi_y$ let
\begin{equation*}
v^\xi_y(t):=v(y+t\xi),\qquad \widehat{v}^\xi_y(t):=v^\xi_y(t)\cdot \xi\,.
\end{equation*}
The following proposition collects some results from \cite{AmbCosDM97} (see Propositions~3.2, 4.7, and Theorem~4.5 therein).
\begin{proposition}\label{prop:1906180944}
Let $v\in L^1(U; \Rn)$ and $e_1,\dots,e_n$ be a basis of $\Rn$. Then $v\in BD(U)$ [resp.\ $SBD(U)$] if and only if for every $\xi=e_i+e_j$, $1\leq i,\, j \leq n$
\begin{equation}\label{2306181022}
\begin{split}
\widehat{v}^\xi_y \in BV(U^\xi_y) \text{ [resp.\ $SBV(U^\xi_y) $]  }\quad &\text{for }\hn\text{-a.e.\ }y\in \Pi^\xi\,,\\
\int \limits_{\Pi^\xi} |\mathrm{D}\widehat{v}^\xi_y|(U^\xi_y) \dh(y) & < +\infty\,. 
\end{split}
\end{equation}
Moreover, let $v\in BD(U)$, $\xi \in \Sn$ and $J^\xi_v:=\{x\in J_v \colon [v]\cdot \xi \neq 0\}$ (it holds that $\hn(J_v\sm J_v^\xi)=0$ for $\hn$-a.e.\ $\xi \in \Sn$). Then 
\begin{equation}\label{1906181642}
\E^j v\, \xi \cdot \xi = \int \limits_{\Pi^\xi} \int\limits_{J_{\widehat{v}^\xi_y}} \big[\widehat{v}^\xi_y\big] \mathrm{d}t \dh(y),\,\quad |\E^j v\, \xi \cdot \xi|(U) = \int \limits_{\Pi^\xi} \int\limits_{J_{\widehat{v}^\xi_y}} \big|\big[\widehat{v}^\xi_y\big] \big| \mathrm{d}t \dh(y),
\end{equation}
and for $\hn$-a.e.\ $y\in\Pi^\xi$
\begin{subequations}
\begin{align}
e(v)^\xi_y\, \xi\cdot \xi &=\nabla \widehat{v}^\xi_y \quad\mathcal{L}^1\text{-a.e.\ on }U^\xi_y\,,\label{3105171927}\\
(J^\xi_v)^\xi_y= J_{\widehat{v}^\xi_y} \quad\text{and}\quad & v^\pm(y+t\xi)  \cdot \xi = (\widehat{v}^\xi_y)^\pm(t) \ \text{ for }t\in (J_v)^\xi_y\,,\label{1906181638}
\end{align}
\end{subequations}
where the normals to $J_v$ and $J_{\widehat{v}^\xi_y}$ are oriented so that $\xi \cdot \nu_v \geq 0$ and $\nu_{\widehat{v}^\xi_y}=1$.
\end{proposition}


For more details on the spaces $BV$, $SBV$ and $BD$, $SBD$ we refer to \cite{AFP} and to \cite{AmbCosDM97, BelCosDM98, Bab15, Tem}, respectively.

\par
\medskip
\paragraph{\bf $GBD$ functions.}
The space $GBD$ of \emph{generalised functions of bounded deformation} has been introduced in \cite{DM13} (to which we refer for a general treatment) and it is defined by slicing as follows. 

\begin{definition}[\cite{DM13}]\label{def:GBD}
Let $\Omega\subset \Rn$ be bounded and open, and $v\colon \Omega\to \Rn$ be $\mathcal{L}^n$-measurable. Then $v\in GBD(\Omega)$ if there exists $\lambda_v\in \mathcal{M}^+_b(\Omega)$ such that the following equivalent conditions hold for every $\xi \in \Sn$:
\begin{itemize}
\item[(a)] for every $\tau \in C^1(\R)$ with $-\tfrac{1}{2}\leq \tau \leq \tfrac{1}{2}$ and $0\leq \tau'\leq 1$, the partial derivative $\mathrm{D}_\xi\big(\tau(v\cdot \xi)\big)=\mathrm{D}\big(\tau(v\cdot \xi)\big)\cdot \xi$ belongs to $\mathcal{M}_b(\Omega)$, and for every Borel set $B\subset \Omega$ 
\begin{equation*}
\big|\mathrm{D}_\xi\big(\tau(v\cdot \xi)\big)\big|(B)\leq \lambda_v(B);
\end{equation*}
\item[(b)] $\widehat{v}^\xi_y \in BV_{\mathrm{loc}}(\Omega^\xi_y)$ for $\hn$-a.e.\ $y\in \Pi^\xi$, and for every Borel set $B\subset \Omega$ 
\begin{equation}\label{3105171445}
\int \limits_{\Pi_\xi} \Big(\big|\mathrm{D} {\widehat{v}}_y^\xi\big|\big(B^\xi_y\setminus J^1_{{\widehat{v}}^\xi_y}\big)+ \mathcal{H}^0\big(B^\xi_y\cap J^1_{{\widehat{v}}^\xi_y}\big)\Big)\dh(y)\leq \lambda_v(B)\,,
\end{equation}
where
$J^1_{{\widehat{v}}^\xi_y}:=\left\{t\in J_{{\widehat{v}}^\xi_y} : |[{\widehat{v}}_y^\xi]|(t) \geq 1\right\}$.
\end{itemize} 
The function $v$ belongs to $GSBD(\Omega)$ if $v\in GBD(\Omega)$ and $\widehat{v}^\xi_y \in SBV_{\mathrm{loc}}(\Omega^\xi_y)$ for every $\xi \in \Sn$ and for $\hn$-a.e.\ $y\in \Pi^\xi$.
\end{definition}
$GBD(\Omega)$ and $GSBD(\Omega)$ are vector spaces, as stated in \cite[Remark~4.6]{DM13}, and one has the inclusions $BD(\Omega)\subset GBD(\Omega)$, $SBD(\Omega)\subset GSBD(\Omega)$, which are in general strict (see \cite[Remark~4.5 and Example~12.3]{DM13}).
Every $v\in GBD(\Omega)$ has an \emph{approximate symmetric gradient} $e(v)\in L^1(\Omega;\Mnn)$, still characterised by \eqref{1906180942} and \eqref{3105171927}, and
the \emph{approximate jump set} $J_v$ is still countably $(\hn,n{-}1)$-rectifiable (\textit{cf.}~\cite[Theorem~6.2]{DM13}) and can be reconstructed by \eqref{1906181638}
 (see \cite[Theorem~8.1]{DM13}).


\par
\medskip
\paragraph{\bf First order differential operators $\A$ and functions of bounded $\A$-variation.} In this paragraph we recall recent results from \cite{BreDieGme17, GmeRai17, GmeRai18}, starting from the notions of $\R$- and $\C$-ellipticity for operators $\A$ of the form \eqref{1806182316}, introduced in Section~\ref{sec:intro}. Such an operator can be seen as $\A u = A(\E u)$, for $A$ endomorphism on $\Mnn$, as in \eqref{2406180949}.

First (see \cite[Theorem~2.6]{BreDieGme17}) $\A$ is $\C$-elliptic if and only if  the kernel of $\A$, defined by
\begin{equation*}
N(\A)=\{v\in \mathcal{D}'(\Rn;\Rn) \colon \A v \equiv 0\}\,,
\end{equation*}
is finite dimensional and contained in the space of polynomials of degree less than $l= l(\A) \in \N$.
\begin{remark}\label{rem:1906181942}
For the symmetrised gradient $\A v=\E v=\frac{1}{2}(\nabla v + \nabla v^T)$, we have $N(\E)=\{ x \mapsto Mx+b \colon M \in \mathbb{M}^{n{\times}n},\, M=-M^T,\, b\in \Rn\}$. For $\A v=\E_D v= \E v -\frac{1}{n}(\mathrm{div\,}v)\, \mathrm{Id}_n$, if $n\geq 3$ this operator is $\C$-elliptic with
\begin{equation*}
N(\E_D)=\{x\mapsto Mx + b + (2(a \cdot x) x - |x|^2 a) \colon M \in \mathbb{M}^{n{\times}n},\, M=-M^T,\,a,\, b \in \Rn\}\,,
\end{equation*}
while, if $n=2$, $\E_D$ is only $\R$-elliptic and $N(\E_D)$ consists of the holomorphic functions, with the identification $\C \cong \R^2$. (The elements of $N(\E_D)$ are usually called conformal Killing vectors, see \cite{Dai06, FucRep11}.)
\end{remark}
By \cite[Lemma~2.3]{BreDieGme17} if $\A$ is $\R$-elliptic there exist $0<\kappa_1 < \kappa_2 < \infty$ such that
\begin{equation}\label{2306181016}
\kappa_1 |w| |z| \leq |w\, \oa z| \leq \kappa_2 |w| |z| \quad \text{for all }w,\,z \in \R^n\,.
\end{equation}

For every open domain $U\subset \Rn$,
the \emph{total $\A$-variation} of $v\in L^1_{\mathrm{loc}}(U; \Rn)$ is (notice that $\A$ is symmetric)
\begin{equation}\label{2206182352}
|\A v|(U):=\sup\Big\{ \int \limits_U v \cdot \A \varphi \dx \colon  \varphi \in C^1_c(U;\Rn),\, |\varphi|\leq 1 \Big\}\,.
\end{equation}
A function $v\in L^1(U;\Rn)$ is of \emph{bounded $\A$-variation} if $|\A v|(U) < \infty$ and we denote
\begin{equation*}
BV^{\A}(U):= \{v \in L^1(U; \Rn) \colon \A v \in \mathcal{M}_b(U; \Mnn)\}\,.
\end{equation*}
The following proposition collects \cite[Theorem~1.3]{VScha13} (see also \cite[Theorem~1.1]{GmeRai17}), \cite[Proposition~4.2 and Lemma~5.8]{GmeRai17},  \cite[Proposition~2.5]{GmeRai18}, and \cite[Theorem~3]{FucRep11}.
\begin{proposition}\label{prop:1906182148}
Let $U$ be bounded and star-shaped with respect to a ball (that is star-shaped with respect to each point of a ball $B\subset U$). If $\A$ is $\C$-elliptic
then there exist a constant $C>0$ such that
\begin{equation}\label{1906182253}
\|v\|_{L^{1^*}(U;\Rn)}\leq C\, \|\A v\|_{L^1(U;\Mnn)}
\end{equation}
for every $v\in C^1_c(U; \Rn)$, and (denoting by $\hookrightarrow$ and $\hookrightarrow\hookrightarrow$  continuous and compact embeddings, respectively)
\begin{equation}\label{2206182156}
\begin{split}
BV^{\A}(U) &\hookrightarrow L^{1^*}(U;\Rn)\\
BV^{\A}(U) &\hookrightarrow\hookrightarrow L^{p}(U;\Rn) \text{ for every }p \in [1, 1^*)\,.
\end{split}
\end{equation}
 If $\A$ is $\R$-elliptic then for every $p\in [1, 1^*)$ there exist $C_p >0$ such that
\begin{equation}\label{1906182255}
\|v\|_{L^{p}(U;\Rn)}\leq C_p \,\|\A v\|_{L^1(U;\Mnn)}
\end{equation}
for every $v \in C^1_c(U; \Rn)$. 
 Moreover, if $\A$ is $\C$-elliptic then there is $C>0$, depending only on $n$, such that for every $v\in BV^{\A}(U)$
\begin{equation}\label{2206182212}
\|v-\pi_U v\|_{L^{1^*}(U)} \leq C  |\A v|(U)\,,
\end{equation} 
for a suitable $\pi_U v  \in N(\A)$. 
If $n=2$, then for every $p\in [1, 1^*)$ there exists $C>0$ depending only on $p$, such that it holds
\begin{equation}\label{2306180932}
\|v-\pi_U \,v\|_{L^{p}(U)} \leq C \, \mathrm{diam}(U)^{1-n+\frac{n}{p}}  |\E_D v|(U)\,,
\end{equation} 
for some $\pi_U \,v  \in N(\E_D)$, 
namely $\pi_U\, v$ is holomorphic (see Remark~\ref{2206181920}).
\end{proposition}

\begin{remark}\label{2206181920}
In \cite{VScha13} it is proven that \eqref{1906182253} is equivalent to the fact that $\A$ is $\R$-elliptic and cancelling, a weaker property than $\C$-ellipticity. For $n=2$, we have that $\E_D$ is only $\R$-elliptic but not cancelling, so only \eqref{1906182255} holds, and $N(\E_D)$ can be identified with the space of holomorphic functions (see \cite[Example~2.4 c)]{GmeRai17}).
\end{remark}

\begin{remark}\label{1912182103}
The estimates \eqref{2206182212} and \eqref{2306180932} may be extended to any connected set $U$ finite union of sets $U_i$ which are bounded and star-shaped with respect to a ball. Indeed, since $N(\A)$ is made of polynomials and due to \eqref{2206182212}, one can find $\pi_U\in N(\A)$ such that $\|\pi_{U_i} v - \pi_U v\|_{L^{1^*}(U)} \leq C |\A v|(U)$ for any $i$. This is true, for $\pi_{U_j}v$ in place of $\pi_U v$, for any $U_i$, $U_j$ with $\mathcal{L}^n(U_i\cap U_j)>0$, by rigidity of polynomials, and it is extended to a finite union. As for \eqref{2306180932}, see \cite[comment before Theorem~3]{FucRep11}.
In particular, one sees that \eqref{2206182212} and \eqref{2306180932} hold if $U$ is a connected Lipschitz domain.
\end{remark}

We prove below the main result of the section.
\begin{theorem}\label{teo:GSBDdiventaSBD}
Let $U\subset \Rn$ be an open 
 bounded 
domain. If $\A$ as in \eqref{1806182316} (i.e., $\A$ symmetric) is $\C$-elliptic, then
\begin{subequations}
\begin{equation}\label{0208181353}
\begin{split}
GBD(U) \cap BV^\A(U) = BD(U)
\end{split}
\end{equation}
and
\begin{equation}\label{0208181354}
\begin{split}
GSBD(U) \cap BV^\A(U) = SBD(U)\,.
\end{split}
\end{equation}
\end{subequations}
\end{theorem}
\begin{proof}
By \eqref{2306181022}, we have that from \eqref{0208181353} one gets \eqref{0208181354}. It is also immediate that $BD(U) \subset GBD(U) \cap BV^\A(U)$, being $\A$ symmetric.  
In order to show the 
opposite inclusion, let us fix $u\in GBD(U) \cap BV^\A(U)$ and first prove (in the spirit of the blow up technique \cite{FonMue92}) that 
\begin{equation}\label{2306180007}
\frac{\dd |\A u|}{\dh \mres J_u} \geq \big| [u] \oa \nu_u \big| \quad \hn\text{-a.e.\ in } J_u\,. 
\end{equation}
 Since $J_u$ is countably rectifiable, so that $\hn\mres J_u$ is $\sigma$-finite, and $\A u \in \mathcal{M}_b(U; \Mnn)$, the Radon-Nikodym derivative of $|\A u|$ with respect to $\hn\mres J_u$ exists (more precisely, it is the function $\theta \in L^1(J_u)$ such that $|\A u^a|= \theta \hn\mres J_u$, where $\A u=\A u^a +\A u^s$, for $\A u^a\ll \hn\mres J_u$, $\A u^s \perp \hn\mres J_u$). Moreover, it may be computed explicitly by (see e.g.\ \cite[Theorems~1.28 and 2.83]{AFP}):
 
\begin{equation}\label{2206182054}
\frac{\dd |\A u|}{\dh \mres J_u}(x)=\lim_{\varrho \to 0} \frac{|\A u|(B_\varrho(x))}{\hn(J_u \cap B_\varrho(x))} \quad\text{for $\hn$-a.e.\ $x \in J_u$}\,.
\end{equation}
For $\hn$-a.e.\ $x \in J_u$, we have also that
\begin{equation}\label{2206182055}
\lim_{\varrho\to 0} \frac{\hn(J_u \cap B_\varrho(x))}{\omega_{n-1}\, \varrho^{n-1}}=1\,,
\end{equation}
for $\omega_{n-1}$ the $n{-}1$-dimensional measure of the unit ball in $\R^{n-1}$, and that, if we 
introduce 
$u_{\varrho,x}(y):=u(x+ \varrho y) \colon B_1(0) \to \R^n$, then (denoting $B:= B_1(0)$)
\begin{equation}\label{2206182056}
\lim_{\varrho\to 0^+} u_{\varrho,x} \to u_0:=u^+(x)\chi_{B^+} + u^-(x) \chi_{B^-}\,\quad \text{in $\mathcal{L}^n$-measure in $B$},
\end{equation}
where $u^\pm(x) \in \Rn$ are the Lebesgue limits at $x$ on the two sides of $J_u$ with respect to $\nu_u(x)$, and $B^{\pm}:=B \cap \{y \in \Rn \colon (y-x) \cdot \nu_u(x) \in \R^\pm\}$  
(see also e.g.\ \cite[Theorem~6.2, below (6.4)]{DM13}). Let us fix $x$ such that these three conditions hold, and denote $u_\varrho \equiv u_{\varrho,x}$.

Since the derivative in \eqref{2206182054} exists finite, by \eqref{2206182055} and the fact that
\begin{equation*}
|\A u_\varrho|(B)= \frac{|\A u|(B_\varrho(x))}{\varrho^{n-1}}\,,
\end{equation*}
we obtain that there exists $C>0$ independent of $\varrho$ such that
\begin{equation}\label{2206182111}
|\A u_\varrho|(B) \leq C\,.
\end{equation}
By the embeddings \eqref{2206182156} we get that $\|u_\varrho\|_{L^1(B)} = \varrho^{-n}\|u\|_{L^1(B_\varrho(x))}< \infty$, so that $u_\varrho \in BV^\A(B)$ for any $\varrho>0$ and \eqref{2206182212}, \eqref{2206182111} imply
\begin{equation}\label{2206182213}
\|u_\varrho - \pi_\varrho\|_{L^{1^*}(B)}\leq C\,,
\end{equation}
where $\pi_\varrho:= \pi_B u_\varrho$. This gives that $(u_\varrho - \pi_\varrho)_\varrho$ is bounded in $BV^{\A}(B)$. Then, by \eqref{2206182156}, up to a (not relabelled) subsequence, $u_\varrho - \pi_\varrho \to \tilde{v} \in \Rn$ a.e.\ in $B$. Recalling \eqref{2206182056}, $\pi_\varrho$ belong to the finite dimensional space of polynomials $N(\A)$ of degree less than $l(\A) \in \N$ (being $\A$ elliptic, cf.\ before Remark~\ref{rem:1906181942}) and converge $\mathcal{L}^n$-a.e.\ in $B$. Therefore
$\pi_\varrho$ converge uniformly to a suitable polynomial $\pi_0$ (indeed, if $\|\pi_\varrho\| \to \infty$, for any norm on the finite dimensional space of polynomials of degree less than $l(\A)$, then $\frac{\pi_\varrho}{\|\pi_\varrho\|}$ converges to a polynomial of degree less than $l(\A)$, so $|\pi_\varrho|$ converges to $+\infty$ up to a $\mathcal{L}^n$-negligible set). 

By difference we obtain that the convergence in \eqref{2206182056} is strong in $L^1(B;\Rn)$, passing 
to a suitable subsequence $\varrho_k$. Looking at the definition of $|\A u|$ in \eqref{2206182352}, we deduce immediately the lower semicontinuity with respect to $L^1$-convergence of $u_\varrho$, so by \eqref{2206182054}, \eqref{2206182055}
\begin{equation*}
\begin{split}
|\A u_0|(B) &\leq \liminf_{k \to \infty} |\A u_{\varrho_k}|(B) = \liminf_{k \to \infty} \frac{|\A u|(B_{\varrho_k}(x))}{(\varrho_k)^{n-1}}\\ &= \lim_{\varrho \to 0} \frac{|\A u|(B_{\varrho}(x))}{\varrho^{n-1}}= \omega_{n-1}  \frac{\dd |\A u|}{\dh \mres J_u}(x)\,.
\end{split}
\end{equation*}
By the special form of $u_0$ (see \eqref{2206182056}), we have directly that
\begin{equation*}
|\A u_0|(B) = \omega_{n-1} \, \big| [u](x) \oa \nu_u(x) \big|\,.
\end{equation*}
This proves the claim \eqref{2306180007}.

Combining \eqref{2306180007} with \eqref{2306181016} (recall that $\A u$ has bounded variation) we obtain that
\begin{equation*}
[u] \in L^1(J_u; \Rn)\,,
\end{equation*}
for $[u]=u^+-u^-$, where $u^\pm$ are the Lebesgue limits in the sense of $GBD$, cf.\ Definition~\ref{def:2306181018}.

It is now possible to fill the gap between the slicing conditions \eqref{3105171445} for $G(S)BD$ and the characterisation of $(S)BD$ functions \eqref{2306181022}, by the area formula for rectifiable sets (see e.g.\ \cite[(12.4) in Section~12]{Sim84}). Since $J_u$ is countably $(\hn, n{-}1)$-rectifiable and $\nu_u \cdot \xi$ is the Jacobian of the projection $p_\xi \colon J_u \to \Pi^\xi$ (we consider $\nu_u \cdot \xi \geq 0$) we obtain for any $\xi \in \Sn$
\begin{equation*}
\begin{split}
\int \limits_{J^\xi_u} \big| [u] \cdot \xi \big|\, (\nu_u \cdot \xi) \dh = \int \limits_{\Pi^\xi} \sum_{t\in (J^\xi_u)^\xi_y} \big| [u] (y+t\xi) \cdot \xi \big| \dh = \int \limits_{\Pi^\xi} \sum_{t\in J_{\widehat{u}^\xi_y}} \big| [\widehat{u}^\xi_y] \big|(t) \dh \,,
\end{split}
\end{equation*} 
recalling that \eqref{1906181638} holds also for $u\in GBD(U)$. Employing \eqref{1906181642} and \eqref{3105171445} in Definition~\ref{def:GBD} (now $\widehat{u}^\xi_y \in SBV_{\mathrm{loc}}(U^\xi_y)$ for $\hn$-a.e.\ $\xi$ if $u\in GSBD(U)$), and the fact that \eqref{3105171927} holds both in $(S)BD$ and $G(S)BD$, we get \eqref{2306181022} and then $u\in BD(U)$.
This concludes the proof.
\end{proof}

\begin{remark}\label{2907181935}
For $n=2$ and $\A=\E_D$ we are not able to deduce that $GBD \cap BV^{\E_D} = BD$ as above, the issue being property \eqref{2306180007} (if this was true, then we would conclude by using \eqref{2306181016}). Indeed, in this case $N(\E_D)$ consists of the holomorphic functions (with the identification $\C \cong \R^2$); employing \eqref{2306180932} instead of \eqref{2206182212} we get \eqref{2206182213} with any fixed $p\in [1, 1^*)$ in place of $1^*$, where $\pi_\varrho$ is holomorphic. Now the problem is that it is not true that the 
$\mathcal{L}^n$-a.e.\ convergence of $\pi_\varrho$ to $\pi_0:= u_0 - \tilde{v}$ takes place also in $L^1$: in general, the convergence is locally uniform just on an open dense subset of $B_1(0)$ (by Osgood's theorem \cite{Osgood1901}).
\end{remark}
%
%
\begin{corollary}\label{cor:2306181058}
If $\A$ is an operator as in Theorem~\ref{teo:GSBDdiventaSBD} and $u\in GBD(U)\cap BV^{\A}(U)$, then
\begin{equation*}
\A u \mres J_u = [u] \oa \nu_u \, \hn \mres J_u\,,
\end{equation*}
applying the operator $\A$ to \eqref{1205181701}.
\end{corollary}

\section{$\Gamma$-$\liminf$ inequality}

 Let us fix a sequence $\varepsilon_k$ and denote by $D_k$ the functionals $D_{\varepsilon_k}$, with analogous notation for all the quantities depending on $\varepsilon$.
We consider an open bounded domain $\Omega'\subset \Rn$ such that $\Omega\subset \Omega'$ and  $\Omega' \cap \dom =\dod$  and set, for each $u$, $v$ defined in $\Omega$, their extensions
\begin{equation*}
\begin{split}
\tilde{u}:=\begin{dcases}
u &\text{ in }\Omega\,,\\
u_0 &\text{ in }\Omega' \sm \Omega\,,
\end{dcases}
\qquad
\tilde{v}:=\begin{dcases}
v &\text{ in }\Omega\,,\\
1 &\text{ in }\Omega' \sm \Omega\,,
\end{dcases}
\end{split}
\end{equation*}
Then we have that
\begin{equation}\label{2506180936}
\begin{split}
& D_k^{\Omega'}(\tilde{u},\tilde{v})- D_k^{\Omega}(u,v) = D^{\Omega'}(\tilde{u},{v})- D^{\Omega}(u,v) \\& \hspace{1em} = \int \limits_{\Omega' \sm \Omega} \Big[ ( 1 + \varepsilon_k^{p-1}) f_p(\A u_0) + ( 1  + \eta_{\varepsilon_k})\, f_p(e(u_0)- \A u_0) \Big] \dx\,, 
\end{split}
\end{equation}
where $D_k^\Omega$, $D^\Omega$ and $D_k^{\Omega'}$, $D^{\Omega'}$ are the functionals $D_k$ and $D$ with the integrals evaluated on $\Omega$ and $\Omega'$. Then it is enough to argue in the enlarged domain $\Omega'$. We denote $\tilde{D}_k:=D^{\Omega'}_k$, $\tilde{D}:= D^{\Omega'}$.
 
%
 First we prove that for given sequences $u_k$, $v_k$ converging in $\mathcal{L}^n$-measure to some $u\colon \Omega\to \Rn$, $v\colon \Omega \to \R$ measurable, such that $D_k(u_k, v_k) < \infty$,   that is 
 \begin{equation}\label{2306181225}
\tilde{D}_k(\tilde{u}_k, \tilde{v}_k) < \infty\,,
 \end{equation}
  (we may assume without loss of generality that 
$D(u_k, v_k)$, and then $\tilde{D}(\tilde u_k,\tilde v_k)$, converges to some finite  limit)   we have
 \begin{equation}\label{2306181226}
 \tilde{u}\in GSBD^p(\Omega') \cap BV^{\A}(\Omega') \quad\text{and} \quad \tilde{v}=1 \text{ a.e.\ in } \Omega'\,.
 \end{equation}
Since $\tilde{D}_k(\tilde{u}_k,\tilde{v}_k) \geq \int \limits_{\Omega} \frac{\psi(\tilde{v}_k)}{\varepsilon_k} \dx$ and $\psi$ is decreasing with $\psi(1)=0$,
we get readily that $\tilde{v}=1$ a.e.\ in $\Omega'$ and $\tilde{v}_k \to 1$ in $\mathcal{L}^n$-measure. As for $u$, recalling the assumptions on $f_p$, we have that for any $\lambda\in [0,1)$ 
\begin{equation}\label{2607181709}
\begin{split}
\tilde{D}_k(\tilde{u}_k, \tilde{v}_k) &\geq C_{f_p}  \int \limits_{\Omega' } \big[ \lambda |A(e(\tilde{u}_k))|^p \chi_{\{ \tilde{v}_k \geq \lambda\}} + \varepsilon_k^{p-1} |A(e(\tilde{u}_k))|^p \chi_{ \{ \tilde{v}_k \leq \lambda\} }\big] \dx \\
& \hspace{3em}+ \frac{\psi(\lambda)}{\varepsilon_k} \mathcal{L}^n( \{ \tilde{v}_k \leq \lambda\} ) - C_{f_p} \, \mathcal{L}^n(\Omega') \\&
\geq \widetilde{C}_{f_p, p} \Big[ \lambda \hspace{-0.7em}\int \limits_{ \{ \tilde{v}_k \geq \lambda\} } \hspace{-0.7em} |A(e(\tilde{u}_k))|^p \dx + p^{\frac{1}{p}}(p')^{\frac{1}{p'}} \psi(\lambda)^{\frac{1}{p'}} \hspace{-0.7em} \int \limits_{ \{ \tilde{v}_k \leq \lambda\} } \hspace{-0.7em} |A(e(\tilde{u}_k))| \dx \Big]  - C_{f_p} \, \mathcal{L}^n(\Omega') \\&
\geq \widetilde{C} \int \limits_{\Omega'} |A(e(\tilde{u}_k))| \dx - \widehat{C}\,,
\end{split}
\end{equation}
for suitable $\widetilde{C}$, $\widehat{C}$ depending on $f_p$, $p$, $\psi$, $\mathcal{L}^n(\Omega)$, and $\lambda$.
Notice that we have employed the operator $A$  in \eqref{2406180949} to underline the dependence on the absolutely continuous part $e(u_k)$ and used the Young inequality for the second estimate above. 

Since $\tilde{u}_k \in L^1(\Omega';\Rn)$, from \eqref{2206182212} and Remark~\ref{1912182103} (we may assume here $\Omega'$ connected, arguing for each connected component of $\Omega$) we get that there are suitable $\pi_{\Omega'} \tilde{u}_k \in N(\A)$ for which 
\[
\|\tilde{u}_k-\pi_{\Omega'} \tilde{u}_k\|_{L^{1^*}(\Omega')} \leq \|\A(e(\tilde{u}_k))\|_{L^1(\Omega')}\,.
\]
By \eqref{2607181709} and the compact embedding in \eqref{2206182156}, up to a subsequence the functions
$\tilde{u}_k-\pi_{\Omega'} \tilde{u}_k$ converge 
strongly in $L^1(\Omega';\Rn)$. Since $\tilde{u}_k$ converge in measure to $\tilde{u}$, then $\pi_{\Omega'} \tilde{u}_k$ converge uniformly to a polynomial in $N(\A)$ (see the proof of Theorem~\ref{teo:GSBDdiventaSBD}), and 
then,
for every $p\in [1,1^*)$,

\begin{equation}\label{2306181618}
\tilde{u}_k \to \tilde{u} \in BV^{\A}(\Omega') \quad\text{ in }L^p(\Omega';\Rn)\,,\qquad \A \tilde{u}_k \wstar \A \tilde{u} \quad\text{ in }\mathcal{M}_b(\Omega';\Mnn)\,.
\end{equation}

Let us now prove that $\tilde{u}$ is in $GSBD^p(\Omega')$, 
employing the two terms depending only on $v$ in $D_k$ to estimate the measure of $J_{\tilde{u}}$. Consider the function $\phi(t):= \int_0^t \psi^{\frac{1}{q'}}(s) \,\dd s$ for $t\in [0,1]$.
By the Young inequality $\frac{\alpha^q}{q}+\frac{\beta^{q'}}{q'}\geq \alpha \, \beta$ for
\begin{equation}\label{2407182322}
\alpha=\big(\gamma\, q \,\varepsilon_k^{q-1} |\nabla v_k|^q \big)^{1/q}\,,\qquad \beta= \big( q' \, \psi(v_k) \,\varepsilon_k^{-1} \big)^{1/q'}\,,
\end{equation}
we find that for any $\lambda \in (0,1)$ 
\begin{equation}\label{2406181703}
\begin{split}
\tilde{D}_k(\tilde{u}_k, \tilde{v}_k)  \geq \hspace{-0.7em}\int \limits_{ \{\tilde{v}_k > \lambda\} } &\hspace{-0.7em}\Big[ \frac{\psi(\tilde{v}_k)}{\varepsilon_k} + \gamma \varepsilon_k^{q-1}|\nabla \tilde{v}_k|^q \Big] \dx  
 \geq q'^{\frac{1}{q'}} (\gamma\, q)^{\frac{1}{q}} \int \limits_{ \{\tilde{v}_k > \lambda\} } \psi^{\frac{1}{q'}}(\tilde{v}_k) |\nabla \tilde{v}_k| \dx 
 \\&
 = q'^{\frac{1}{q'}} (\gamma\, q)^{\frac{1}{q}} \int\limits_{ \{\tilde{v}_k > \lambda\} } | \nabla(\phi(\tilde{v}_k))| \dx \\
 & = q'^{\frac{1}{q'}} (\gamma\, q)^{\frac{1}{q}} \int_{\phi(\lambda)}^{\phi(1)} \hn(\partial^* \{\phi(\tilde{v}_k) > s\} ) \dd s\,,
\end{split}
\end{equation}
employing the Coarea formula for $\phi(\tilde{v}_k)$.
Therefore, fixed $\lambda \in (0,1)$,
for any $\lambda' \in (\lambda,1)$ the Mean Value theorem guarantees the existence of $\tilde{\lambda}_k \in (\lambda, \lambda')$ such that (notice that $\phi$ is strictly increasing)
\begin{equation*}
\begin{split}
\hn&(\partial^*\{\tilde{v}_k > \lambda_k \}) = \hn(\partial^*\{\phi(\tilde{v}_k) > \phi(\lambda_k)\}) \\& \leq (\phi(\lambda') - \phi(\lambda))^{-1} \int _{\phi(\lambda)}^{\phi(\lambda')} \hn(\partial^* \{\phi(\tilde{v}_k) > s\} ) \dd s < C\,.
\end{split}
\end{equation*}
It follows that the functions $\widehat{u}_k:= \tilde{u}_k \,\chi_{ \{\tilde{v}_k > \lambda_k \} }$ satisfy
\begin{equation*}
\E \widehat{u}_k = e(\tilde{u}_k) \,\chi_{ \{\tilde{v}_k > \lambda_k \} } \mathcal{L}^n \,+\, \tilde{u}_k \odot \nu_{\partial^*\{ \tilde{v}_k > \lambda_k\}  } \hn \mres \partial^* \{\tilde{v}_k > \lambda_k \} \,.
\end{equation*}
Since $\lambda_k \geq \lambda>0$, we get a control for $e(\tilde{u}_k)$ in $L^p$, and with the estimate above this gives  
\begin{equation}\label{2306181656}
\int \limits_\Omega |e(\widehat{u}_k)|^p \dx + \hn(J_{\widehat{u}_k}) \leq C\,,
\end{equation}
and \[\mathcal{L}^n(\{ \tilde{u}_k \neq \widehat{u}_k\})= \mathcal{L}^n(\{ \tilde{v}_k \leq \lambda_k\}) \leq \mathcal{L}^n(\{ \tilde{v}_k \leq \lambda'\}) \leq \varepsilon_k \frac{\tilde{D}_k(\tilde{u}_k, \tilde{v}_k)}{\mathcal{L}^n(\Omega)\, \psi(\lambda')} \to 0\,,\] so that
\begin{equation}\label{2306181703}
\widehat{u}_k \to \tilde{u} \quad\text{ in $\mathcal{L}^n$-measure in }\Omega'\,.
\end{equation}
 By \eqref{2306181656} and \eqref{2306181703} (this latter condition implies that there exists a continuous function $\tilde{\psi}$ diverging to $+\infty$ such that $\int_{\Omega'} \tilde{\psi}(\tilde{u}_k) \dx < C < +\infty$), we may apply \cite[Theorem~11.3]{DM13}
 (or we may use the compactness theorem for $GSBD$ \cite[Theorem~1.1]{CC18}, since the exceptional set $A$ therein is empty by \eqref{2306181703})  to get
\begin{equation}\label{2406181254}
\begin{split}
\tilde{u}\in GSBD^p(\Omega')\,, \qquad  e(\widehat{u}_k) \weak e(\tilde{u}) \text{  in }L^p(\Omega'; \Mnn)\,. 
\end{split}
\end{equation}
Together with \eqref{2306181618} this proves the claim \eqref{2306181226}.
At this stage Theorem~\ref{teo:GSBDdiventaSBD} implies that
\begin{equation*}
\tilde{u} \in SBD^p(\Omega')\,.
\end{equation*}
By the weak convergences \eqref{2406181254}, the fact that $\tilde{v}_k$ converge to 1 uniformly up to a set of vanishing measure, 
%
and the Ioffe-Olech semicontinuity theorem, see e.g.\ \cite[Theorem~2.3.1]{ButLibro}, we get that for any $\lambda\in (0,1)$ (cf.\ \cite[(4.4) in proof of Theorem~3.3]{FocIur14} and \cite[(5.4a) in proof of Theorem~5.1]{CC17})
\begin{equation}\label{2306181808}
\begin{split}
&\int \limits_{\Omega'} \Big[ f_p\big(A(e(\tilde{u}))\big) + f_p\big(e(\tilde{u})-A(e(\tilde{u}))\big) \Big] \dx \\& \hspace{1em}\leq \liminf_{k\to \infty} \int \limits_{ \{ \tilde{v}_k >\lambda\} } \Big[ (\tilde{v}_k+ \varepsilon_k^{p-1}) f_p(\A \tilde{u}_k) + (\tilde{v}_k + \eta_{\varepsilon_k}) \, f_p(e(\tilde{u}_k)- \A \tilde{u}_k) \Big] \dx\,.
\end{split}
\end{equation}
As for the Ambrosio-Tortorelli term $\int \limits_\Omega \big[ \frac{\psi(\tilde{v}_k)}{\varepsilon_k} + \gamma \varepsilon_k^{q-1}|\nabla \tilde{v}_k|^q \big] \dx$ in $\tilde{D}_k$, by a standard argument 
(see e.g.\ \cite[(4.18)]{FocIur14}, now we argue in the enlarged domain $\Omega'$)  we obtain 
\begin{equation*}\label{2406181316}
\liminf_{k\to \infty} \hn(\partial^* \{ \phi(\tilde{v}_k) > s\} ) \geq 2 \hn(J_{\tilde{u}})= 2 \hn \big(J_u \cup (\dod \cap \{ \mathrm{tr}(u-u_0) \neq 0 \} )\big)
\end{equation*}
for every $s \in (\phi(\lambda), \phi(1))$. Together with \eqref{2406181703} this gives
\begin{equation}\label{2406181705}
\begin{split}
& 2 (q')^{1/q'} (\gamma q)^{1/q} \big(\phi(1)-\phi(\lambda)\big) \hn \big(J_u \cup (\dod \cap \{ \mathrm{tr}(u-u_0) \neq 0 \} )\big) \\& \hspace{1em}\leq \liminf_{k\to \infty}  \int \limits_{ \{\tilde{v}_k > \lambda\} } \hspace{-0.7em}\Big[ \frac{\psi(\tilde{v}_k)}{\varepsilon_k} + \gamma \varepsilon_k^{q-1}|\nabla \tilde{v}_k|^q \Big] \dx \,.
\end{split}
\end{equation}

Let us now estimate the other significant term in the limit by
\begin{equation}\label{2506180931}
\int \limits_{ \{\tilde{v}_k \leq \lambda\} } \Big[ (\tilde{v}_k+ \varepsilon_k^{p-1}) f_p(\A \tilde{u}_k) + \frac{\psi(\tilde{v}_k)}{\varepsilon_k} \Big] \dx \geq p^{\frac{1}{p}} (p')^{\frac{1}{p'}} \psi(\lambda)^{\frac{1}{p'}} \int \limits_{ \{\tilde{v}_k \leq \lambda\} }  (f_p)^{\frac{1}{p}}(\A \tilde{u}_k)  \dx\,,
\end{equation}
thanks to the Young inequality.
We claim that for any $\lambda \in (0, 1)$
\begin{equation}\label{2406181747}
\int \limits_{J_{\tilde{u}}}  (\tilde{f}_p)^{\frac{1}{p}}\big( [\tilde{u}] \oa \nu_u \big) \dh \leq \liminf_{k\to \infty} \int \limits_{ \{\tilde{v}_k \leq \lambda\} }  (f_p)^{\frac{1}{p}}(\A\tilde{u}_k) \dx \,.
\end{equation}
Up to a subsequence, that we do not relabel, we may assume that the $\liminf$ above is a limit, so it is enough to 
prove \eqref{2406181747} 
 along any further   subsequence.
Let us introduce the positive measures defined on any $B\subset \Omega'$ Borel set by
\begin{equation*}
\mu_k(B):= \hspace{-1.3em} \int \limits_{B\cap \{\tilde{v}_k \leq \lambda\} } \hspace{-1.3em}(f_p)^{\frac{1}{p}}(\A \tilde{u}_k) \dx\,,\qquad \widehat{\mu}_k(B):= \int \limits_{B} (\tilde{f}_p)^{\frac{1}{p}}(\A \tilde{u}_k) \dx\,.
\end{equation*}
By \eqref{2306181618} we get that $\mu_k$ and $\widehat{\mu}_k$ are equibounded, and then (up to a subsequence)
\begin{equation*}
\mu_k \wstar \mu\,,\qquad \widehat{\mu}_k \wstar \widehat{\mu} \quad \text{ in } \mathcal{M}_b^+(\Omega')\,.
\end{equation*}
Therefore we want to prove that the Radon-Nikodym derivatives of $\mu$ and $\widehat{\mu}$ satisfy
\begin{equation}\label{2406181758}
\frac{\dd \mu}{\dh \mres J_{\tilde{u}}} = \frac{\dd \widehat{\mu}}{\dh \mres J_{\tilde{u}}} \geq (\tilde{f}_p)^{\frac{1}{p}}\big( [\tilde{u}] \oa \nu_{\tilde{u}} \big) \quad \text{$\hn$-a.e.\ in } J_{\tilde{u}}\,.
\end{equation}
With \eqref{2406181758} at disposal, we conclude \eqref{2406181747} since then 
\[
 \int \limits_{J_{\tilde{u} \cap B}} (\tilde{f}_p)^{\frac{1}{p}}\big( [\tilde{u}] \oa \nu_{\tilde{u}} \big) \dh \leq \mu(B) 
\] 
as (positive) measures on $\Omega'$, and 
\[
\mu(\Omega') \leq \liminf_{k\to \infty} \mu_k(\Omega')= \liminf_{k\to \infty} \int \limits_{ \{\tilde{v}_k \leq \lambda\} }  (f_p)^{\frac{1}{p}}(\A\tilde{u}_k) \dx \,.
\]
In order to show \eqref{2406181758},  we argue in the spirit of \cite[Proof of~(4.6)]{FocIur14} (the functions giving the density of elastic energy are there supposed to be quadratic in $e(u)$, we include the case where these have $p$-growth and are not $p$-homogeneous, cf. \eqref{1806181046}).  
Let us define the measures $\zeta_k \in \mathcal{M}_b^+(\Omega')$ by 
\begin{equation}\label{2506180012}
\zeta_k(B):= D_k^B(\tilde{u}_k, \tilde{v}_k)\,,\quad\text{for }B\subset \Omega' \text{ Borel,} 
\end{equation}
where $D_k^B$ denotes the localisation of $D_k$ to the set $B$. By \eqref{2306181225}, $\zeta_k$ are equibounded in $\mathcal{M}_b^+(\Omega')$, so, up to a subsequence, $\zeta_k \wstar \zeta \in \mathcal{M}_b^+(\Omega')$.
Recalling Corollary~\ref{cor:2306181058}  and since 
 $(\tilde{f}_p)^{\frac{1}{p}}$ is a norm,  
we have that 
\begin{equation}\label{2506180052}
\frac{\dd\, (\tilde{f}_p)^{\frac{1}{p}}(\A \tilde{u})}{\dh \mres J_{\tilde{u}}} = (\tilde{f}_p)^{\frac{1}{p}}([\tilde{u}] \oa \nu_{\tilde{u}}) \quad\hn\text{-a.e.\ in } J_{\tilde{u}}\,.
\end{equation}
Let us fix $x \in J_{\tilde{u}}$ such that the derivatives in \eqref{2406181758} plus $\frac{\dd \zeta}{\dh \mres J_{\tilde{u}}}$ exist finite in $x$, and \eqref{2506180052} is verified in $x$ 
(this holds for $\hn$-a.e.\ $x\in J_{\tilde{u}}$); let
\begin{equation*}
I:= \{ \varrho \in (0, \mathrm{d}(x, \partial \Omega') ) \colon \mu(\partial B_\varrho(x))= \widehat{\mu}(\partial B_\varrho(x))= \zeta(\partial B_\varrho(x))=0   \}\,.
\end{equation*}
For every $\varrho \in (0, \mathrm{d}(x, \partial \Omega') )$
consider the three sets (that partition $B_\varrho(x)$)
 \begin{equation*}
\begin{split} 
 & E_1:=B_\varrho(x) \cap \{ |\A\tilde{u}_k| \leq \varrho^{-\frac{1}{2}}\} \cap \{\tilde{v}_k \leq \lambda\}\,, 
 \\&E_2:=B_\varrho(x) \cap \{ |\A\tilde{u}_k| > \varrho^{-\frac{1}{2}} \} \cap \{\tilde{v}_k \leq \lambda\}\,,\\& 
 E_3:=B_\varrho(x) \cap \{\tilde{v}_k > \lambda\}\,.
 \end{split}
 \end{equation*} 
Since, by \eqref{2607182230}, there exists $C'_{f_p}\geq C_{f_p}>0$ such that $f_p(\xi)\leq C'_{f_p} (1+ |\xi|^p)$ 
and $\tilde{f}_p(\xi)\leq C'_{f_p} |\xi|^p$, 
it holds that
\begin{equation}\label{0912180024}
\int \limits_{E_1} (f_p)^{\frac{1}{p}}(\A\tilde{u}_k) \dx \leq C'_{f_p} \, (\omega_{n-1} \varrho^n + \varrho^{n-\frac{1}{2}})\,, \quad \int \limits_{E_1} (\tilde{f}_p)^{\frac{1}{p}}(\A\tilde{u}_k) \dx \leq C'_{f_p} \, \varrho^{n-\frac{1}{2}}\,.
\end{equation}
By \eqref{1806181046} 
We have that
\begin{equation}\label{0912180022}
\int \limits_{E_2} \big| (f_p)^{\frac{1}{p}}(\A\tilde{u}_k) - (\tilde{f}_p)^{\frac{1}{p}}(\A\tilde{u}_k) \big| \dx  \leq \delta_{\varrho} \hspace{-1em} \int \limits_{B_\varrho(x)\cap \{\tilde{v}_k \leq \lambda\} } \hspace{-1em} | \A \tilde{u}_k| \dx \leq C(C_{f_p})  \,\delta_{\varrho} \, \mu_k(B_\varrho(x))  \,,
\end{equation}
for 
\[
\delta_\varrho:=\hspace{-1em}\sup_{s> \varrho^{-1/2},\, |\xi|=1} \hspace{-0.2em}\Bigg|\frac{(f_p)^{\frac{1}{p}}\big(s \xi\big)}{|s|}- (\tilde{f}_p)^{\frac{1}{p}}\big(\xi\big)  \Bigg|\,,
\]
using \eqref{2607182230} and the fact that
 \[
\sup_{s> \varrho^{-1/2}} \Bigg|\frac{(f_p)^{\frac{1}{p}}\big(s \frac{\A \tilde{u}_k}{|\A \tilde{u}_k|}\big)}{|s|}- (\tilde{f}_p)^{\frac{1}{p}}\big(\frac{\A \tilde{u}_k}{|\A \tilde{u}_k|}\big)  \Bigg| \leq \delta_\varrho\,. \] 
By \eqref{1806181046}, $\lim_{\varrho\to 0} \delta_\varrho=0$ (uniformly in $k$). 
Thus, the estimate \eqref{0912180022} and the fact that $\lim_\varrho \lim_k \varrho^{-(n-1)} \mu_k(B_\varrho(x)) < C$ (by the choice of $\varrho$ and $x$, in particular $\frac{\dd \mu}{\dh \mres J_{\tilde{u}}}$ exists finite at $x$) give that
\begin{equation}\label{0912180023}
\lim_{\substack{\varrho \in I \\ \varrho\to 0}} \lim_{k\to \infty} \varrho^{-(n-1)}  \int \limits_{E_2} \big| (f_p)^{\frac{1}{p}}(\A\tilde{u}_k) - (\tilde{f}_p)^{\frac{1}{p}}(\A\tilde{u}_k) \big| \dx =0\,.
\end{equation}
On the other hand H\"older's inequality gives
 \begin{equation}\label{2506180034}
 \begin{split}
 \int \limits_{E_3} (f_p)^{\frac{1}{p}}(\A\tilde{u}_k) \dx &\leq \Big(\int \limits_{E_3} f_p(\A\tilde{u}_k) \dx\Big)^{\frac{1}{p}} \big( \mathcal{L}^n(E_3) \big)^{\frac{1}{p'}} \\& 
 \leq \lambda^{-\frac{1}{p}}\big(\zeta_k(B_\varrho(x))\big)^{\frac{1}{p}} \big( \mathcal{L}^n(B_\varrho(x)) \big)^{\frac{1}{p'}}\,,
 \end{split}
 \end{equation}
 since 
 \[
 \zeta_k(B_\varrho(x))=D_k^{B_{\varrho}(x)}(\tilde{u}_k, \tilde{v}_k)\geq \int \limits_{E_3} (\tilde{v}_k + \varepsilon_k^{p-1}) f_p(\A\tilde{u}_k) \dx \geq \lambda \int \limits_{E_3} f_p(\A\tilde{u}_k)  \dx\,.
 \]
 Therefore we obtain
 \begin{equation}\label{2506180109}
 \begin{split}
\frac{\dd \mu}{\dh \mres J_{\tilde{u}}} = \lim_{\substack{\varrho \in I \\ \varrho\to 0}} \lim_{k\to \infty} \frac{\mu_k(B_\varrho(x))}{\omega_{n-1} \varrho^{n-1}}= \lim_{\substack{\varrho \in I \\ \varrho\to 0}} \lim_{k\to \infty} \frac{\widehat{\mu}_k(B_\varrho(x))}{\omega_{n-1} \varrho^{n-1}}= \frac{\dd \widehat{\mu}}{\dh \mres J_{\tilde{u}}}\,.
 \end{split}
 \end{equation}
 Indeed, the first and the last equalities follow by definition of Radon-Nikodym derivative and the choice of $I$, while the central equality descends by putting together \eqref{0912180024}, \eqref{2506180034} (divided by $\omega_{n-1} \varrho^{n-1}$), and \eqref{0912180023}. 
In order to deal with \eqref{2506180034}, we remark that $\lim_\varrho \lim_k \varrho^{-(n-1)} \zeta_k(B_\varrho(x)) < C$ since $\frac{\dd \zeta}{\dh \mres J_{\tilde{u}}}$ exists finite at $x$.

Since $\widehat{\mu}_k$ is defined in terms of the convex positively 1-homogeneous $(\tilde{f}_p)^{\frac{1}{p}}$ and  $\A\tilde{u}_k \wstar \A \tilde{u}$ in $\mathcal{M}_b(\Omega';\Mnn)$  by \eqref{2306181618}, Reshetnyak Semicontinuity Theorem (see e.g.\ \cite[Theorem~2.38]{AFP}) implies that 
\begin{equation*}
(\tilde{f}_p)^{\frac{1}{p}}(\A \tilde{u}) (B_\varrho(x)) \leq \liminf_{k\to \infty} (\tilde{f}_p)^{\frac{1}{p}}(\A\tilde{u}_k) (B_\varrho(x)) = \liminf_{k\to \infty} \widehat{\mu}_k(B_\varrho(x)) = \widehat{\mu}(B_\varrho(x)) \,,
\end{equation*}
if $\varrho \in I$. Taking the Radon-Nikodym derivative of the above inequality with respect to $\hn \mres J_{\tilde{u}}$ at $x$, for $I \ni \varrho \to 0$,   and recalling \eqref{2506180109} and the choice of $x$ (that gives in particular \eqref{2506180052} at $x$), we deduce  \eqref{2406181758} and then prove the claim \eqref{2406181747}.

We now collect \eqref{2306181808}, \eqref{2406181705}, \eqref{2506180931}, \eqref{2406181747} and use the arbitrariness of $\lambda \in (0,1)$ (indeed we let $\lambda\to 0$) to conclude the $\Gamma$-$\liminf$ inequality
\begin{equation*}
\tilde{D}(\tilde{u}, \tilde{v}) \leq \liminf_{k\to \infty} \tilde{D}_k(\tilde{u}_k, \tilde{v}_k)\,,
\end{equation*}
that gives the desired inequality $D(u,v) \leq \liminf_{k\to \infty} D_k(u_k, v_k)$, by \eqref{2506180936}.

Moreover, notice that \eqref{2607181709} gives also the inclusion stated in Theorem~\ref{teo:main} for the sublevels of $D_\varepsilon$. 
 The corresponding compactness property 
follows arguing as done for proving \eqref{2306181618}, but now the boundedness of the polynomials $\pi_{\Omega'} \tilde{u}_k$ is a consequence of the fact that $\tilde{u}_k = u_0$ in $\Omega' \sm \ol \Omega$ (we argue separately on each connected component, using Remark~\ref{1912182103}). The convergence of quasi-minimisers for $D_\varepsilon$ to a minimiser for $D$ follows by general properties of $\Gamma$-convergence (see e.g.\ \cite[Corollary~7.17]{DMLibro}).

\begin{remark}\label{rem:0208181957}
If $n=2$ and $\A=\E_D$, by \eqref{2607181709} and \eqref{1906182255} we get still \eqref{2306181618}, as well as \eqref{2406181254}, arguing as done for $n \geq 3$. If we had at disposal the analogous of Theorem~\ref{teo:GSBDdiventaSBD} (and then Corollary~\ref{cor:2306181058}) we could follow the proof of $\Gamma$-$\liminf$ inequality as above.  
\end{remark}

\begin{remark}\label{rem:0208182013}
We could reproduce the proof of the $\Gamma$-$\liminf$ inequality above for $n\geq 3$ and the operator
\begin{equation*}
\mathbb{B} u := \E_D u + \frac{\mathrm{div}^+ u}{n}\mathrm{Id}\,.
\end{equation*}
Indeed $BV^{\mathbb{B}}(\Omega') \subset BV^{\E_D}(\Omega')$, so that $GSBD(\Omega') \cap BV^{\mathbb{B}}(\Omega')=SBD(\Omega')$,  and then, applying $\mathbb{B}$ to \eqref{1205181701},
\begin{equation}\label{0408180835}
\mathbb{B} u \mres J_u = \Big[ \big( [u] \odot \nu_u \big)_D +   \frac{([u] \cdot \nu_u)^+}{n}\mathrm{Id} \Big] \hn \mres J_u =: [u] {\otimes_{\mathbb{B}}} \nu_u\,.
\end{equation}
Moreover, it holds $\mathbb{B} \tilde{u}_k \wstar \mathbb{B} \tilde{u}$ in $\mathcal{M}_b(\Omega'; \Mnn)$.
 Thus we get \eqref{2506180052} for $[u] {\otimes_{\mathbb{B}}} \nu_u$, and so the corresponding version of \eqref{2406181747}. 
\end{remark}
\section{$\Gamma$-$\limsup$ inequality}
As in \cite{FocIur14}, we construct by hand a recovering sequence starting from a function $u$ with regular jump set and smooth outside $J_u$. However, since our result is formulated for general $SBD$ functions without requiring \emph{a priori} integrability for $u$ it is not enough now to apply neither the density results for $GSBD$ \cite[Theorem~3.1]{Iur14} and 
\cite{FriPWKorn, CFI17Density, CC17}, nor the approximations  \cite{Cha04, Cha05Add} for $SBD$. Indeed all these results do not approximate the jump part of $u$ without assuming $u \in L^\infty(\Omega;\Rn)$: this request  is not natural because the functionals, that depend on $e(u)$, are not decreasing by truncation of $u$. 

The analysis is then based on the following approximation for $SBD^p$ functions in $BD$-norm, recently proven in \cite[Theorem~1.1]{Cri19}.
\begin{theorem}\label{teo:density}
Let $\Omega$ be an open bounded Lipschitz subset of $\Rn$, and $u\in SBD^p(\Omega)$, with $p>1$.
Then there exist $u_k\in SBV^p(\Omega;\Rn)\cap L^\infty(\Omega; \Rn)$ such that each
$J_{u_k}$ is closed 
and included in a finite union of closed connected pieces of $C^1$ hypersurfaces, $u_k\in C^\infty(\ol \Omega\sm J_{u_k};\Rn) \cap W^{m,\infty}(\Omega\setminus J_{u_k}; \Rn)$ for every $m\in \N$, and:
\begin{equation}\label{1main}
\lim_{k\to \infty}\Big(\|u_k-u\|_{BD(\Omega)}+\|e(u_k) - e(u)\|_{L^p(\Omega;\Mnn)} + \hn(J_{u_k}\triangle J_u)\Big) = 0\,.
\end{equation}
\end{theorem}

We combine the previous approximation with a well-known result by Cortesani and Toader, that 
allows us to work with the so-called ``piecewise smooth''  $SBV$-functions, denoted $\mathcal{W}(\Omega;\Rn)$, namely
\begin{equation*}
u \in \mathcal{W}(\Omega;\Rn) \text{ if }\begin{cases}
u\in SBV(\Omega;\Rn)\cap W^{m,\infty}(\Omega\sm J_u;\Rn) \,\text{for every }m\in \N\,,\\
\hn(\ol{J_u} \sm J_u ) = 0\,,\\
\ol{J_u} \text{ is the intersection of $\Omega$ with a finite union of ${(n{-}1)}$-dimensional simplexes}\,.
\end{cases}
\end{equation*}
We report below the result by Cortesani and Toader, in a slightly less general version.
\begin{theorem}[\cite{CorToa99}, Theorem~3.1] \label{teo:CorToa}
Let $\Omega$ be an open bounded Lipschitz set.
For every $u\in SBV^p(\Omega;\Rn) \cap L^\infty(\Omega;\Rn)$ there exist $u_k\in \mathcal{W}(\Omega;\Rn)$ such that
\begin{align*}
\lim_{k\to \infty} \Big( \|u_k-u\|_{L^1(\Omega;\Rn)} &+ \|\nabla u_k -\nabla u\|_{L^p(\Omega;\mathbb{M}^{n\times n})} + \hn(J_{u_k}\triangle J_u)  \Big)=0\,,\\
\lim_{k \to \infty} \int \limits_{J_{u_k}\cap A} \phi(x, u_k^+, & u_k^-, \nu_{u_k})   \dh = \int \limits_{J_u \cap A} \phi(x, u^+, u^-, \nu_{u}) \dh\,,
\end{align*}
for every $A\subset \Omega$, $\hn(\partial A \cap J_u)=0$, and every $\phi$ strictly positive, continuous, and $BV$-elliptic (see e.g.\ \cite{Amb90GSBV} or \cite[equation (2.4)]{CorToa99} for the notion of $BV$-ellipticity).
\end{theorem}

\begin{remark}\label{rem:2506181133}
In Theorem~\ref{teo:CorToa} we may assume also $J_{u_k} \Subset \Omega$, by \cite[Remark~6.3]{Cri19}, in turn using \cite{DPFusPra17}. At this stage, \cite[Lemma~5.2]{DPFusPra17} gives that for any $p>1$ the $n{-}1$ dimensional simplexes in the decomposition of $\ol J_u$ may be taken \emph{pairwise disjoint} and such that also $J_u \cap \Pi_j \cap \Pi_i=\emptyset$ for any two different hyperplanes $\Pi_i$, $\Pi_j$ (if $p \in (1,2]$ it is enough to employ the capacitary argument in \cite[Remark~3.5]{CorToa99}).
Moreover, we notice that our function $(\tilde{f}_p)^{\frac{1}{p}}$ is $BV$-elliptic.
\end{remark}

The combination of the density results described so far guarantees that for a given $u \in SBD^p(\Omega)$ we can find approximating functions $u_k \in \mathcal{W}(\Omega,\Rn)$ with $J_{u_k}\Subset \Omega$ and $J_{u_k} \cap \Pi_j \cap \Pi_i=\emptyset$ for any two different hyperplanes $\Pi_i$, $\Pi_j$. The last property we have to ensure is that 
\begin{equation}\label{2212181528}
\mathrm{tr}_{\dom}\, u_k = \mathrm{tr}_{\dom}\, u_0 \quad \text{ on }\dod\,.
\end{equation} 
This is possible in view of the assumption \eqref{1806181920}, 
 arguing as in \cite[Theorem~5.5]{CC17} with tools from \cite{Cri19}, as sketched below. 
 
 For given $u\in SBD^p(\Omega)$ and $\varepsilon>0$,  one first defines a suitable extension $\widehat{u}_k$ of $u$ on $\Omega_t:=\Omega+B(0,t)$, for $t< 32 k^{-1}$, 
as follows.  We can find pairwise disjoint cubes $(Q_h)_{h=1}^{\ol h}$, centered at $x_h \in \don$ with sidelength $\varrho_h$,  $\mathrm{d}(Q_h, \dod)> d_\varepsilon>0$ (recall \eqref{3101190848}), 
 \begin{equation}\label{2012182236}
 \int\limits_{\don \sm \widehat{Q}} \hspace{-0.5em}1 +|\tr_{\dom}(u-u_0)| \dh < \varepsilon, \qquad D^{\Omega \cap \widehat{Q}}(u,1) < \eta_\varepsilon\,,\qquad\text{for }\widehat{Q}:=\bigcup_h \ol Q_h\,,
 \end{equation}
 for suitable $d_\varepsilon$, $\eta_\varepsilon$ with $\lim_{\varepsilon\to 0} d_\varepsilon=\lim_{\varepsilon\to 0} \eta_\varepsilon=0$  ($D^A$ denotes the energy $D$ in Theorem~\ref{teo:main} localised on a set $A$),
 $J_u \cap  \partial Q_h = \emptyset$ for each $h$, 
 \begin{equation}\label{1701191821}
 u \in L^1\Big(\Omega \cap \bigcup_h \partial Q_h;\Rn\Big)\,,\qquad u_0 \in L^1\Big(\bigcup_h \partial Q_h;\Rn\Big)\,. 
\end{equation} 
 Moreover, $\don \cap Q_h$ is ``almost'' a diameter of $Q_h$ with respect to $\nu_\dom(x_h)$, namely (cf. \cite[(4.2)]{Cri19}) there exists $C^1$ hypersurfaces $(\Gamma_h)_{h=1}^{\ol h}$ with $x_h \in \Gamma_h \subset Q_h$ and
\begin{equation}\label{2212181446}
\begin{split}
 \hn\big((\don\triangle \Gamma_h)\,\cap &\, \ol{Q}_h\big)< \varepsilon (2\varrho_h)^{n-1}< \, \frac{\varepsilon}{1-\varepsilon}  \hn(\don\cap \ol{Q}_h)\,, \\
 \Gamma_h\, \text{is a $C^1$ graph with respect} & \text{ to }  \nu_{\dom}(x_h) \text{ with Lipschitz constant less than $\varepsilon/2$. }
 \end{split}
 \end{equation}
  Let
\[
\widetilde{u}:=u \chi_\Omega + u_0 \chi_{\Omega_t \sm \Omega}\,,
\]
and notice that by \eqref{2212181446} we can say that  (up to modify $\eta_\varepsilon$) 
\begin{equation}\label{3001191103}
\int\limits_{\don \sm \Gamma_h} |[u-u_0]| \dh = \int\limits_{\don \sm \Gamma_h} |[\widetilde{u}]| \dh < \eta_\varepsilon\,.
\end{equation} 
We now approximate $\widetilde{u}$ with respect to the energy $D$, arguing in each $Q_h$, by a sequence of functions, depending on a parameter $k$.
 We notice that the choice of the finite family of cubes $Q_h$ is done before the construction of these approximations, and depends only on $\varepsilon$.
Then we can argue, as follows, for a fixed cube, denoting $Q\equiv Q_h$, $\Gamma \equiv \Gamma_h$ and assuming, up to a rotation and a translation, $x_h=0$ and $\nu_{\dom}(x_h)=e_n$ (notice that all the notation indeed depends on $h$). Let $Q^-$ denote the almost half cube contained in $Q$ which is below $\Gamma$ (that is $Q^-$ is almost contained in $\Omega$).

 We partition $Q^-$ in parallelepipeds with first $n{-}1$ coordinates in squares of sidelength $(\eta_\varepsilon\, k)^{-1}$ 
 \begin{equation*}
F_{\textbf{m}}:=\big\{ (y_1,\dots, y_{n-1})\in \R^{n-1}\colon y_i \in (\eta_\varepsilon k)^{-1} m_i + \big(0, (\eta_\varepsilon k)^{-1})\big\}
\end{equation*} 
(we have $\textbf{m}=(m_1,\dots,m_{n-1})\in \{-\eta_\varepsilon k \varrho, -\eta_\varepsilon k \varrho+1, \dots, 0, \dots, \eta_\varepsilon k \varrho-1\}^{n-1} \subset \N^{n-1}$, we may assume $\eta_\varepsilon k \varrho\in \mathbb{N}$)  so that
  \begin{equation*}
  \Gamma_h \cap (F_\m {\times} \R) \subset F_\m{\times} (m_n, m_n +1/2)k^{-1}\,,
  \end{equation*}
  for some $m_n \in \N$ (cf.\ \cite[(4.7), (4.8)]{Cri19}). 
As in \cite[(4.9)]{Cri19}, setting $
Q^-_\m:=Q^- \cap (F_\m{\times}\R)$
we use the Nitsche-type extension \cite[Lemma~2.1]{Cri19}(see also \cite{Nit81}) to extend $\widetilde{u}|_{Q^-_\m}$ 
along the vertical direction, employing the (part of) hyperplans $F_\m{\times}\{m_n\, k^{-1}\}$ 
as the flat interface needed in \cite[Lemma~2.1]{Cri19}: we obtain a function $\widetilde{u}^-_\m$ defined on 
\[
\widetilde{Q}^-_\m:= Q\cap \big(F_\m{\times} (-\infty, (m_n+33)k^{-1})\big)= F_\m{\times} (-\varrho, (m_n+33)k^{-1})\,,
\] 
with $\widetilde{u}^-_\m = \widetilde{u}$ on $F_\m{\times} (-\varrho, (m_n-33)k^{-1})\big)$.
\begin{figure}[h]\label{fig}
\hspace{-2em}
\begin{minipage}[c]{0.49\linewidth}
\includegraphics[width=\linewidth]{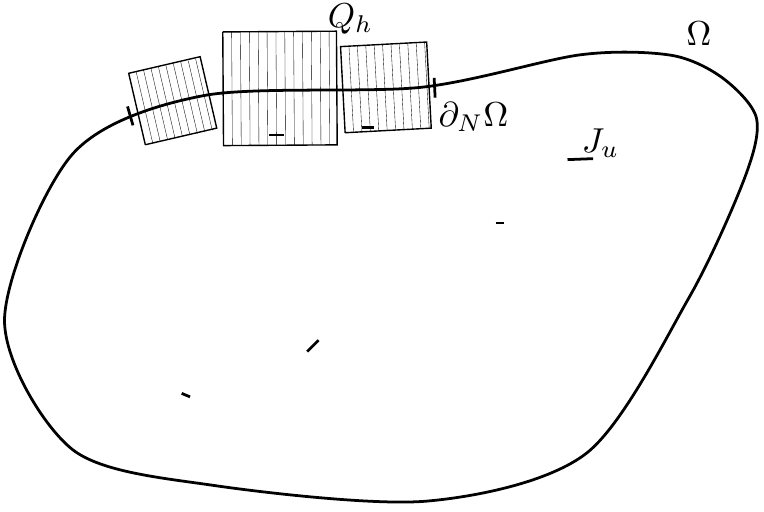}
\end{minipage}
\hfill
\begin{minipage}[c]{0.53\linewidth}
\includegraphics[width=\linewidth]{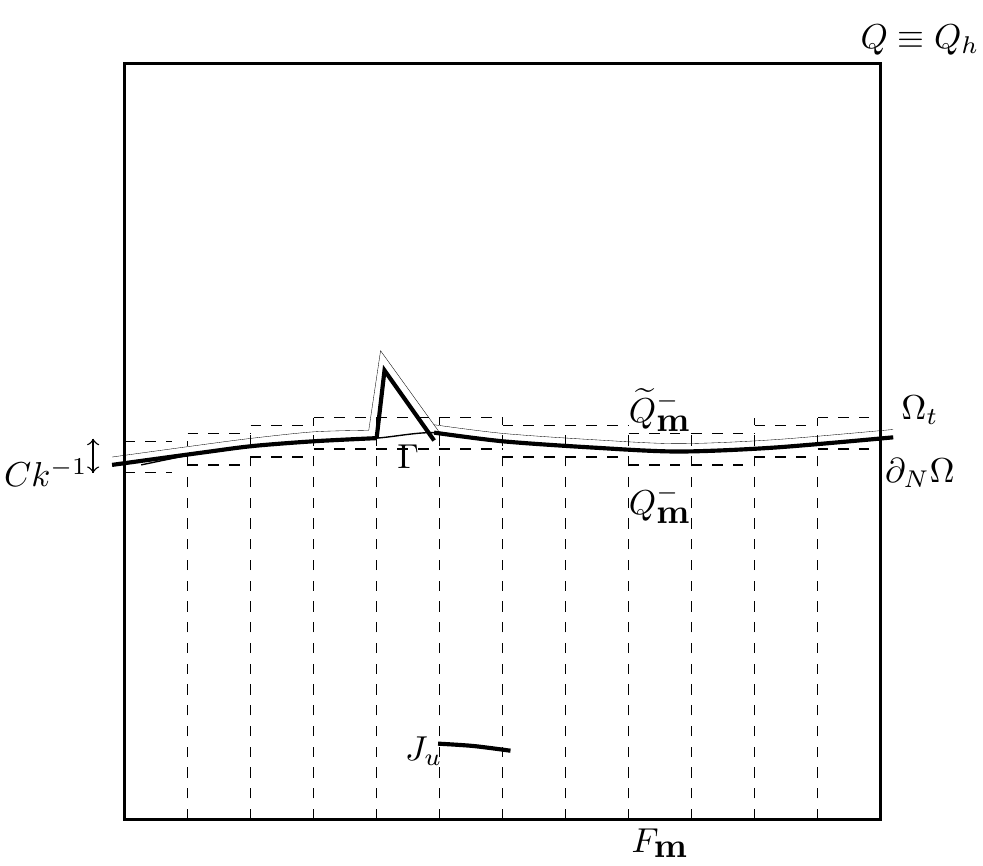}
\end{minipage}
\caption{In the first figure, the cubes $Q_h$ covering almost all $\don$. In the second one, a single cube $Q\equiv Q_h$ with the relative (almost) parallelepipeds $Q^-_\m$, their bottom faces $F_\m$, and their extensions $\widetilde{Q}^-_\m$.
We see the enlarged domain $\Omega_t$, the $C^1$ almost diameter $\Gamma$, and the pieces of hyperplanes $F_\m{\times}\{m_n k^{-1}\}$, below $\Gamma$, along which the original function is extended into $\widetilde{u}^-_\m$. The zones in which we extend \emph{à la} Nitsche have height of order $k^{-1}$.} 
\end{figure}
We do not create a jump on the common boundary between adjacent $Q^-_\m$ except for a region with height of order $k^{-1}$, and this is true also for the jump created with respect to the original $\widetilde{u}$ on the ``boundary parallelepipeds'', namely the $Q^-_\m$ with $\partial Q^-_\m \cap \partial Q \neq \emptyset$.
With the same arguments of \cite[Section~4]{Cri19}, one can control both the measure of the union of these small interfaces (by $C \eta_\varepsilon \varrho^{n-1}$, see \cite[(4.27)]{Cri19}), and the integral of the jump amplitude over this set (cf.\ \cite[(4.32)]{Cri19}).

%
We obtain that, for a universal $c>1$, (neglect the boundary contribution in $D$)
\begin{equation}\label{2212181447}
D^{\widetilde{Q}_\m^-}(\widetilde{u}^-_\m,1) < c D^{Q_\m^-}(\widetilde{u},1)\,,\quad  \quad D^{\widetilde{Q}_\m^- \cup \widetilde{Q}_{\m'}^-}(\widetilde{u}^-_\m \chi_{\widetilde{Q}_\m^-}+ \widetilde{u}^-_{\m'}\chi_{\widetilde{Q}_{\m'}^-},1) < c D^{Q_\m^-\cup Q_{\m'}^-}(\widetilde{u},1)\,,
\end{equation}
 for adjacent $\m$, $\m'$.
 Notice that, since the extension is done with respect to the vertical direction, for the ``boundary parallelepipeds'' $Q^-_\m$ we have that
 \begin{equation}\label{1701191839}
\|\tr \widetilde{u}^-_\m\|_{L^1(\{\dd(\cdot, \dom) < t\} \cap  \partial Q)} \leq c \|u\|_{L^1(\Omega \cap \{\dd(\cdot, \dom) < 2t\} \cap  \partial Q)}\,,
\end{equation}
which vanishes as $k\to \infty$, for $\varepsilon>0$, by \eqref{1701191821} (this is true also taking the union of $\partial Q_h$, since $Q_h$ are in finite number, independent of $k$). 
Eventually, since $u\in BD$, we are able to estimate the trace of $\widetilde{u}^-_\m$ on $F_\m{\times} \{ (m_n+33)k^{-1})\}$, 
in terms of the trace of $u$ on $\Gamma$ (cf.\ e.g.\ \cite[(4.35)]{Cri19}); then we can say that, if $\ol \Omega_t$ intersects  $F_\m{\times} \{ (m_n+33)k^{-1})\}$ (as in the corner for $\don$ in the Figure~1), then
\begin{equation}\label{3001191219}
\int\limits_{\ol\Omega^t \sm \big(F_\m{\times} \{ (m_n+33)k^{-1})\}\big)} \hspace{-4em}|\tr \widetilde{u}^-_\m - u_0| \dh < \int\limits_{\don \sm \Gamma_h} \hspace{-1em} |[\widetilde{u}]| \dh + o_{k\to \infty}(1)\,.
\end{equation}
By \eqref{2012182236}--\eqref{3001191219}
we can see (cf.\ again \cite[(4.16)--(4.34)]{Cri19}) that the extension
\begin{equation*}
\widehat{u}_k:=\begin{dcases}
u &\quad\text{in }\Omega \sm \widehat{Q}\,,\\
 \widetilde{u}^-_\m  &  \quad\text{in }\widetilde{Q}_\m^-\,,\text{ for any }\m \text{ and any } h\,,\\ 
u_0 &\quad \text{elsewhere in }\Omega_t\,.
\end{dcases}
\end{equation*} 
satisfies (in the following neglect the boundary contribution in $D$ when this is evaluated on $\Omega_t$, $\Omega_{t/2}$, in other words we treat the boundary outside $\dod$ as a Neumann part)
\begin{equation*}
D^{\Omega_t}(\widehat{u}_k,1) < D^{\Omega}(u,1)+ c\, D^{\Omega\cap \widehat{Q}}(u,1) +\int\limits_{\don \sm (\cup_h\Gamma_h)} \hspace{-1em}|[\widetilde{u}]| \dh + o_{k\to \infty}(1)\,. 
\end{equation*}
 The vanishing term $o_{k\to \infty}(1)$ accounts also for the jump created on $\{\dd(\cdot, \dom) < 2t\}  \cap \bigcup_h \partial Q_h$, that is controlled by \eqref{1701191821} and \eqref{1701191839}, due to the choice of the cubes $Q_h$. 
At this stage, we can follow the strategy of \cite[Theorem~5.5]{CC17}. For $\delta=t/2=16 k^{-1}$, the function
\begin{equation*}
\widehat{u}_k':=\widehat{u}_k \circ (O_{\delta,x_0})^{-1} + u_0 - u_0 \circ  (O_{\delta,x_0})^{-1}
\end{equation*}
is equal to $u_0$ in a neighbourhood of $\dod$, by \eqref{1806181920}, and satisfies
\begin{equation*}
D^{\Omega_{t/2}}(\widehat{u}_k',1)< D^{\Omega_t}(\widehat{u}_k,1) + o_{k\to \infty}(1)\,.
\end{equation*}
Then we apply the construction of \cite[Theorem~1.1]{Cri19} to $\widehat{u}_k'$, to get $\widetilde{u}_k$ with
\begin{equation*}
D^{\Omega}(\widetilde{u}_k, 1) < D^{\Omega_{t/2}}(\widehat{u}_k',1) + o_{k\to \infty}(1)\,,
\end{equation*}
and $\widetilde{u}_k=u_0\ast \varrho_k$ in a neighbourhood of $\dod$, for $\varrho_k$ a convolution kernel at scale $k^{-1}$. Eventually, we obtain the approximating function $u_k$, satisfying \eqref{2212181528} and close in energy to $u$, by
\begin{equation*}
u_k:= \widetilde{u}_k + u_0 - u_0 \ast \varrho_k\,.
\end{equation*}

%

We are therefore allowed to start (employing Theorem~\ref{teo:CorToa}, that preserves the boundary condition, in a neighbourhood of $\dod$) from a function $u \in \mathcal{W}(\Omega; \Rn)$ with $u=u_0$ in a neighbourhood of $\dod$, $J_u \Subset \Omega$, and it is not restrictive to consider the case $J_u \subset \Pi$ for a suitable hyperplane $\Pi$, say $\Pi=\{x_n=0\}$ to fix a simple notation.
(From now on we regard $x\in \Rn$ as $(x', x_n)$ for $x'\in \R^{n-1}$.)

\begin{remark}\label{rem:2212181552}
To get \eqref{2212181528} it is enough to assume  \eqref{1806181920} separately for $\dod \cap \Omega_j$ and suitable $x_0^j \in \Rn$, $\ol \delta^j>0$, for each connected component $\Omega_j$ of $\Omega$. Moreover, if $p\leq 1^*=\frac{n}{n{-}1}$, the hypothesis \eqref{1806181920} may be dropped by using partitions of the unity 
to guarantee the condition \eqref{2212181528}. This is possible since $u \in L^{1^*}(\Omega;\Rn)$, and so $e(\varphi u)=\varphi \, e(u) + u \odot \nabla \varphi$ is well controlled in $L^p$ for any smooth $\varphi$. We also refer  to \cite{CarVG18} for a corresponding treatment of smooth domains.
\end{remark}   

\begin{remark}\label{rem:2507182324}
A variant of Theorem~\ref{teo:density} with a strong approximation of $\A u$ in $L^p$ and $e(u)-\A u$ in $L^t$ for functions in $SBD^{p\wedge t}$ would allow us to prove the result for functionals $D_k$ depending on $e(u)-\A u$ through a function $g_t$ with $t$-growth, $t\neq p$. Unfortunately, following the proof of Theorem~\ref{teo:density}, this would follow from a refined version of \cite[Proposition~3]{CCF16} controlling two different powers of $\A u$ and $e(u)- \A u$: this seems out of reach with the strategy in \cite{CCF16} that relies on slicing properties, useless for $\A u$. 
For this reason we take $p$ growth both in $\A u$ and in $e(u)-\A u$ (we could consider two different functions $f_p$ and $g_p$, but it is almost the same, 
 taking $f_p$ that acts very differently in the two cases).
\end{remark}

Let us now construct 
 a recovery sequence corresponding to a regular $u$, in the sense described above, by adapting the argument in \cite[Theorem~3.4]{FocIur14}. 
We set
\begin{equation}\label{2207181254}
 \sigma_k(x):=\frac{\varepsilon_k}{2\,p' \psi^{\frac{1}{p'}}(0)} p^{\frac{1}{p}} (p')^{\frac{1}{p'}} (\tilde{f}_p)^{\frac{1}{p}}(|[u](x', 0) \oa e_n|) \quad \text{for }x\in J_u= J_u \cap \Pi\,.
 \end{equation}
 Since $u$ is Lipschitz up to $J_u$, then also $\sigma_k$ is Lipschitz with 
 \begin{equation}\label{2606180829}
 |\nabla \sigma_k| \leq C \varepsilon_k\,,
 \end{equation}
where $C>0$ depends on the Lipschitz constant of $u$, $f_p$, $\A$. As in \cite{FocIur14}, let for any $\varrho<1$
\begin{equation*}
h_1(\varrho):= \psi(1-\varrho),\quad h_2(\varrho):= \Big( \int_0^{1-\varrho} \psi^{-\frac{1}{q}}(s) \dd s \Big)^{-1}, \quad h(\varrho):= h_1\, h_2 (\varrho)\,.
\end{equation*}
Since $\psi$ is positive and vanishing in 1, we have that $h$ is increasing and vanishing in 0 
and $\frac{h_1}{h_2}$ also vanishes in 0, so that $\varrho_k:= h^{-1}(\varepsilon_k)$ is vanishing and
\begin{equation}\label{2606181051}
\lim_{k\to \infty} \frac{h_1(\varrho_k)}{\varepsilon_k} = \lim_{k \to \infty} \frac{\varepsilon_k}{h_2(\varrho_k)}=0\,.
\end{equation}
Let $w_k$ be the unique solution to the Cauchy problem
\begin{equation*}
\begin{dcases}
w'_k= \Big( \frac{q'}{\gamma q} \Big)^{\frac{1}{q}} \varepsilon_k^{-1} \psi^{\frac{1}{q}}(w_k)\,,\\
w_k(0)=0\,,
\end{dcases}
\end{equation*}
in $[0, T_k)$, where $T_k:= \big(\frac{\gamma q}{q'}\big)^{\frac{1}{q}} \varepsilon_k \int_{0}^1 \psi^{-1/q}(s) \dd s  \in (0, \infty]$.
We have that $w_k$ is the inverse of the function 
\[
z\in (\varepsilon_k, 1]\mapsto \Big(\frac{\gamma q}{q'}\Big)^{\frac{1}{q}} \varepsilon_k \int_{0}^{z} \psi^{-1/q}(s) \dd s\,,
\] 
in $[0, T_k)$.
Let $\tau_k:=w_k^{-1}(1-\varrho_k)$, namely 
\begin{equation*}
\tau_k =  \Big(\frac{\gamma q}{q'}\Big)^{\frac{1}{q}} \varepsilon_k \int_{0}^{1-\varrho_k} \psi^{-1/q}(s) \dd s \in (0, T_k)\,,
\end{equation*}
which is infinitesimal in view of \eqref{2606181051}, and define the sets 
\begin{equation*}
\begin{split}
A_k &:= \{ x \in \Rn \colon (x',0) \in J_u, \, |x_n|< \sigma_k(x')\}\,,\\
B_k &:= \{ x \in \Rn \colon (x',0) \in J_u, \, 0 \leq |x_n| - \sigma_k(x') \leq \tau_k\}\,,\\
C_k &:= \{ x \in \Rn \colon (x',0) \notin J_u, \, \mathrm{d}(x, J_u) \leq \tau_k\}\,.
\end{split}
\end{equation*}
The candidate recovery sequence $(u_k, v_k)$ is then
\begin{equation*}
u_k(x):=
\begin{dcases}
\frac{x_n + \sigma_k(x')}{2 \sigma_k(x')} \big( u(x', \sigma_k(x')) - u(x', -\sigma_k(x')) \big) + u(x', -\sigma_k(x'))\,,&\quad\text{if }x\in A_k\,,\\
u(x)&\quad\text{if }x\notin A_k
\end{dcases}
\end{equation*}
and (recall that in the functional there are $v+ \varepsilon_k$ and $v+ \eta_{\varepsilon_k}$)
\begin{equation*}
v_k(x):= 
\begin{dcases}
0 &\quad \text{if }x\in A_k\,,\\
w_k(|x_n|-\sigma_k(x'))  &\quad \text{if }x\in B_k\,,\\
w_k(\mathrm{d}(x, J_u))  &\quad \text{if }x\in C_k\,,\\
1-\varrho_k  &\quad \text{otherwise.}\\
\end{dcases}
\end{equation*}
It is immediate that the sequences $(u_k)_k$ and $(v_k)_k$ converge pointwise to $u$ and $1$.
Moreover, for the components $u_k^i$ of $u_k$,
\begin{equation}\label{2207181238}
\partial_n u_k^i(x)= \frac{u^i(x', \sigma_k(x')) - u^i(x', -\sigma_k(x'))}{2 \sigma_k(x')}\,,\qquad i=1,\dots,n\,,
\end{equation}
and, by straightforward calculations 
(see also \cite{FocIur14}) 
\begin{equation}\label{2207181212}
\begin{split}
|\partial_j u_k^i(x)|\leq |\partial_j \sigma_k(x')|\Big( \frac{|[u^i](x',0)|}{2 \sigma_k(x')} + 4 L \Big) + 3L &\leq C \quad\text{for }j=1,\dots,n{-}1\,,\\
|\partial_n u_k^i(x)| \leq L + \frac{|[u^i](x',0)|}{2 \sigma_k(x')} &\leq \frac{C}{\varepsilon_k}\,,
\end{split}
\end{equation}
in $A_k$, where $L$ is the Lipschitz constant of $u$ in $\Omega\sm J_u$ and $C$ depends on $L$ (recall also \eqref{2606180829}). By the way, $u_k$ is a Lipschitz function. Notice also that
\begin{equation}\label{2207181213}
\lim_{k\to \infty} |\partial_n u_k^i(x)|=\infty\qquad\text{for }x\in J_u\,.
\end{equation}

Let us estimate the energy $D_k(u_k, v_k)$.
We have that
\begin{equation*}
\limsup_{k\to \infty} \int \limits_{\Omega\sm A_k} (v_k+\varepsilon_k^{p-1}) f_p(\A u_k) \dx = \limsup_{k\to \infty} \int \limits_{\Omega\sm A_k} (v_k+\varepsilon_k^{p-1}) f_p(\A u) \dx \leq \int \limits_{\Omega} f_p(\A u) \dx\,,
\end{equation*} 
and
\begin{equation*}
\limsup_{k\to \infty} \int \limits_{\Omega\sm A_k} (v_k+\eta_{\varepsilon_k}) f_p\big(e(u_k) -\A u_k\big) \dx \leq \int \limits_\Omega  f_p\big(e(u)- \A u\big) \dx\,.
\end{equation*} 
Now, recalling \eqref{2207181212}, we get
\begin{equation*}
\int \limits_{A_k} (v_k+\eta_{\varepsilon_k}) f_p\big(e(u_k) -\A u_k\big) \dx =  \int \limits_{A_k} \eta_{\varepsilon_k} f_p\big(e(u_k) -\A u_k\big) \dx \leq \mathcal{L}^n(A_k) \eta_{\varepsilon_k} \frac{C^p}{(\varepsilon_k)^p} \leq C \frac{\eta_{\varepsilon_k}}{(\varepsilon_k)^{p-1}}\,,
\end{equation*}
and this tends to 0 since $\lim_{\varepsilon\to 0} \frac{\eta_\varepsilon}{\varepsilon^{p-1}}=0$.

In view of the fact that $\lim_{k\to \infty}\mathcal{L}^n(A_k)=0$ and of the estimates for the tangential derivatives \eqref{2207181212}, we do not see the contribution of the tangential derivatives in the limit. Moreover, \eqref{2207181213} and assumption \eqref{1806181046} allows us to replace $f_p$ with $\tilde{f}_p$ for the normal derivatives, in the limit. Then (see \eqref{1906181746})
\begin{equation}\label{0308182238}
\limsup_{k\to \infty} \int \limits_{A_k} (v_k+\varepsilon_k^{p-1}) f_p(\A u_k) \dx = \limsup_{k\to \infty} \int \limits_{A_k} \varepsilon_k^{p-1} \tilde{f}_p(\partial_n u_k \oa e_n) \dx\,,
\end{equation}
for $\partial_n u_k$ the vector of the normal derivatives in \eqref{2207181238}
and $e_n=(0,\dots,0,1)\in \Rn$.
By \eqref{2207181238} and the fact that $\tilde{f}_p$ is positively $p$-homogeneous we get 
\begin{equation*}\label{2207181233}
\begin{split}
 \limsup_{k\to \infty} \int \limits_{A_k}&  (v_k+\varepsilon_k^{p-1}) f_p(\A u_k) \dx \\& = \limsup_{k\to \infty} \int \limits_{A_k} \frac{\varepsilon_k^{p-1}}{(2\, \sigma_k(x'))^{p}} \tilde{f}_p\big((u(x', \sigma_k(x')) - u(x', -\sigma_k(x')))\oa e_n\big) \dh(x') \,.
\end{split}
\end{equation*}
Recalling the definitions 
of $\sigma_k$ and $A_k$, the pointwise convergence of $u(x', \sigma_k(x')) - u(x', -\sigma_k(x'))$ to $[u](x)\equiv [u](x')$, a change of variables and the Dominated Convergence Theorem give that the terms above are equal to
\begin{equation*}
\begin{split}
\int \limits_{J_u} \lim_{k\to \infty} \frac{\varepsilon_k^{p-1}}  {(2\, \sigma_k(x'))^{p-1}} \tilde{f}_p\big([u](x')\oa e_n\big) \dh(x') 
\end{split}
\end{equation*}
and then to
\begin{equation*}
\frac{p^{1/p} (p')^{1/p'} \psi(0)^{1/p}}{p}\int \limits_{J_u} (\tilde{f}_p)^{\frac{1}{p}}([u]\oa e_n)\dh\,,
\end{equation*}
since 
\[
(p')^{(1-1/p')(p-1)}=(p')^{1/p'}\,,\qquad p^{-\frac{p-1}{p}}=p^{1/p}/p\,.
\]
As for the remaining terms of $D_k$, notice that by the definition of $v_k$ and \eqref{2606181051} we have
\begin{equation*}\label{2307180129}
\begin{split}
\limsup_{k\to \infty}\int \limits_\Omega \Big[\frac{\psi(v_k)}{\varepsilon_k} + \gamma \varepsilon_k^{q-1}|\nabla v_k|^q \Big] \dx \leq &\limsup_{k\to \infty}\int \limits_{B_k \cup \,C_k} \Big[\frac{\psi(v_k)}{\varepsilon_k} + \gamma \varepsilon_k^{q-1}|\nabla v_k|^q \Big] \dx  \\& +\limsup_{k\to \infty}\int \limits_{A_k} \frac{\psi(v_k)}{\varepsilon_k} \dx\,.
\end{split}
\end{equation*}
We deduce now that
\begin{equation*}\label{2307180130}
\begin{split}
\limsup_{k\to \infty}\int \limits_{B_k} & \Big[\frac{\psi(v_k)}{\varepsilon_k} +  \gamma \varepsilon_k^{q-1}|\nabla v_k|^q \Big] \dx  = \limsup_{k\to \infty} \int \limits_{J_u} \Big( \int_0^{\tau_k}  \Big[\frac{\psi(w_k)}{\varepsilon_k} +  \gamma \varepsilon_k^{q-1} (w'_k)^q \Big]  \mathrm{d} x_n \Big) \, \dh(x') \\
 & = 2 (q')^{1/q'} (\gamma q)^{1/q} \limsup_{k\to \infty} \int \limits_{J_u} \Big( \int_0^{\tau_k} \psi^{1/q'}(w_k) \, w'_k \, \mathrm{d}x_n \Big) \dh(x')
\\  & = 2 (q')^{1/q'} (\gamma q)^{1/q} \limsup_{k\to \infty} \Big(\int_{0}^{1-\varrho_k} \psi^{1/q'}(s)\, \mathrm{d}s\Big) \, \hn(J_u)\\& = a \,\hn(J_u)
\,.
\end{split}
\end{equation*}
Indeed in the first equality we have used the estimate \eqref{2606180829} to neglect the contribution of the tangential derivatives of $v_k$ in the limit, and the second one follows from the definition of $w_k$ ($w'_k$ represents the normal derivative of $v_k$) that gives $\alpha^q=\beta^{q'}$,  that is the condition to have the Young equality $\frac{\alpha^q}{q}+\frac{\beta^{q'}}{q'}= \alpha \, \beta$,  for (recall \eqref{2407182322})
\begin{equation*}
\alpha=\big(\gamma\, q \,\varepsilon_k^{q-1} (w_k')^q \big)^{1/q}\,,\qquad \beta= \big( q' \, \psi(w_k) \,\varepsilon_k^{-1} \big)^{1/q'}\,.
\end{equation*}
%
Furthermore, arguing similarly and using the Coarea formula (cf.\ \cite[eq.\ (4.49)]{FocIur14})  we get 
\begin{equation*}
\int \limits_{C_k} \Big[\frac{\psi(v_k)}{\varepsilon_k} + \gamma \varepsilon_k^{q-1}|\nabla v_k|^q \Big] \dx \leq C  \mu_k \int_{0}^{1-\varrho_k} \psi^{1/q'}(s)\, \mathrm{d}s \leq C \mu_k\,,
\end{equation*}
so that
\begin{equation*}\label{2307180131}
\limsup_{k\to \infty} \int \limits_{C_k} \Big[\frac{\psi(v_k)}{\varepsilon_k} + \gamma \varepsilon_k^{q-1}|\nabla v_k|^q \Big] \dx=0\,.
\end{equation*}
Eventually
\begin{equation*}\label{2307180132}
\int \limits_{A_k} \frac{\psi(v_k)}{\varepsilon_k} \dx = \int \limits_{J_u} \frac{2 \sigma_k(x')}{\varepsilon_k} \psi(0)\dh(x')= \frac{p^{1/p} (p')^{1/p'} \psi(0)^{1/p}}{p'} \int\limits_{J_u} (\tilde{f}_p)^{\frac{1}{p}}([u]\oa e_n)\dh\,.
\end{equation*}
Collecting all the estimates below 
 \eqref{2207181213} we then conclude the $\Gamma$-$\limsup$ inequality.
\begin{remark}\label{rem:0308182234}
With the notation of Remark~\ref{rem:0208182013}, we could reproduce also the proof of the $\Gamma$-$\limsup$ inequality for $\mathbb{B}$ in place of $\A$. Indeed, we define $\sigma_k$ and $u_k$ in terms of $\mathbb{B}$,
and notice that in \eqref{0308182238} we see in the limit $(\partial_n u_k \odot e_n)_D$ plus the contribution of $(\partial_n u_k^n)^+$, asymptotically equal to that of $\mathrm{div}^+ u_k$ by \eqref{2207181212}. Now $2\sigma_k (\partial_n u_k \odot e_n)_D$ converge pointwise to $([u]\odot e_n)_D$ and $2\sigma_k (\partial_n u_k^n)^+$ 
to $[u^n]^+=([u]\cdot e_n)^+$, which gives $\mathbb{B}u$ in the limit, according to \eqref{0408180835}.
\end{remark}

%
%
 
%
%


\bigskip
\noindent {\bf Acknowledgements.}
Vito Crismale has been supported by the Marie Sk\l odowska-Curie Standard European Fellowship No.\ 793018, and by a public grant as part of the \emph{Investissement d'avenir} project, reference ANR-11-LABX-0056-LMH, LabEx LMH.


\end{document}